\documentclass[12pt,a4paper,reqno]{amsart}
\usepackage[utf8]{inputenc}
\usepackage{enumerate}
\usepackage[margin=2.5cm]{geometry}
\usepackage{graphicx} 
\usepackage{overpic}
\usepackage{caption}
\usepackage{subcaption}
\usepackage{xcolor}
\usepackage{comment}
\usepackage{hyperref}
\usepackage{amsthm}
\usepackage{placeins}
\usepackage{float}
\usepackage{enumerate}

\usepackage{lineno} 

\newtheorem{lema}{Lemma}[section]
\newtheorem{tho}[lema]{Theorem}
\newtheorem{pro}[lema]{Proposition}
\newtheorem{remark}[lema]{Remark}

\newtheorem{coro}[lema]{Corollary}

\newtheorem{definition}[lema]{Definition}
\DeclareMathOperator{\ee}{e}

\title[Global analysis of a resource-consumer model] {Global dynamics and regime shifts in a resource-consumer model with facilitation and habitat loss}

\author[T. Mayayo]{Teodoro Mayayo}
\address{Departament de Matem\`{a}tiques, Universitat Aut\`{o}noma de Barcelona, 08193 Be\-lla\-ter\-ra, Barcelona (Spain)}
\email{teodoro.mayayo@uab.cat}

\author[J. Sardany\'{e}s]{Josep Sardany\'{e}s}
\address{Centre de Recerca Matem\`{a}tica, Campus de Be\-lla\-ter\-ra, 08193 Bellaterra, Barcelona (Spain)}
\email{jsardanyes@crm.cat}

\author[J. Torregrosa]{Joan Torregrosa}
\address{Departament de Matem\`{a}tiques, Universitat Aut\`{o}noma de Barcelona, 08193 Be\-lla\-ter\-ra, Barcelona (Spain); Centre de Recerca Matem\`{a}tica, Campus de Be\-lla\-ter\-ra, 08193 Bellaterra, Barcelona (Spain)}
\email{joan.torregrosa@uab.cat}

\makeatletter

\begin{document}

\begin{abstract} 
Modelling how populations respond to habitat loss is crucial for understanding ecosystem stability, especially when positive interactions among resource species, such as plant-plant facilitation, play a key role. Habitat loss not only reduces available organic nutrients and space for primary producers but also disrupts the positive feedbacks that sustain resource populations, thereby affecting consumer persistence and the overall system's stability. We analyse a cubic planar model describing resource-consumer dynamics with facilitation under progressive habitat loss. Our study characterizes the parameter space and enumerates all the phase portraits within the Poincar\'e disk under ecologically relevant conditions. We show that the system has a unique stable limit cycle and characterize analytically the heteroclinic bifurcation curve involving the collapse of the resource and the consumer, enabling us to determine how the parameter region sustaining coexistence oscillations narrows under habitat destruction. To further explore these dynamics, we construct a piecewise-linear (PWL) approximation that preserves the system's qualitative behaviour, allowing us to obtain an explicit expression for the heteroclinic bifurcation. Finally, we investigate how extrinsic noise affecting the resource species impacts the overall dynamics, showing that stochasticity can anticipate the onset of the heteroclinic bifurcation causing earlier co-extinctions.
\end{abstract}

\maketitle

\section{Introduction}
Ecosystems operate through complex, nonlinear interactions, e.g., competition, predation, and/or facilitation, that can sometimes drive sudden shifts between alternative states. When these tipping points are present, slight alterations in key parameters or in environmental conditions can trigger major and often irreversible transformations~\cite{Dakos2019,Sardanyes2024,Scheffer2009,Scheffer2001a,Sole2024}. The preservation of healthy ecological communities emerges from the interplay of multiple factors acting at distinct spatial and temporal scales~\cite{Sole2024}. At the local scale, habitat loss and fragmentation are major drivers of biodiversity decline, altering the structure and function of ecosystems across biomes \cite{Fahrig2003,Haddad2015,Laurance2010}. As natural areas are degraded or fragmented, carrying capacities are reduced, interaction networks are altered, and species persistence and thus biodiversity is threatened~\cite{Brooks2002,Haddad2015,Hanski2005}. Habitat loss-induced extinctions have been described for several species, including the Northern Spotted Owl, which has suffered steep population declines due to old-growth forest logging as well as due to interspecific competition processes~\cite{FRANKLIN2021109168}. The Fringed Leaf Frog, formerly restricted to the Atlantic Forest near São Paulo, Brazil, is believed to be extinct following extensive habitat loss and degradation due to deforestation and urban expansion~\cite{Silvano2005}. 

Understanding species' responses to habitat loss is challenging due to the nonlinear nature of ecological dynamics~\cite{Hanski2005,Sardanyes2024,Sole2024}. This is especially important in species experiencing fluctuations, early identified in natural populations~\cite{Costantino1997,Dennis1997,Elton1924a, Elton1924b,Schaffer1984}. Perturbing a given trophic level can cause a cascading effect: the depletion of a resource species, e.g., primary producers, can undermine consumer populations, and the loss of consumers can in turn alter resource dynamics. Empirical studies, both from long-term field observations and experimental manipulations, have documented the cascading consequences (e.g., extinction debt: how delayed extinctions follow habitat loss~\cite{Tilman1994}) of habitat destruction on species richness, community dynamics, and ecosystem services~\cite{Gonzalez1998,Tilman1994}. These impacts are especially pronounced in fragmented landscapes, where the loss of connectivity and the decline in habitat quality interact to amplify extinction risks~\cite{Fahrig2003,Hanski2011}.

Parallel to empirical efforts, theoretical models have played a crucial role in identifying general principles governing species persistence and extinction under habitat loss. Classic metapopulation frameworks~\cite{Hanski1991,Levins1969} and spatially explicit models~\cite{Hanski2011,Ovaskainen2004,Sardanyes2019,With1999} have provided insights into thresholds of habitat availability and connectivity necessary for population viability. More recent theoretical work has expanded to include further nonlinear ecological interactions, such as competition and predation~\cite{Keitt1997}, yet facilitation--a key mechanism in many ecological contexts--has received comparatively less attention in models of habitat destruction.

Importantly, many resource species--especially plants in terrestrial habitats--exhibit facilitation. Facilitation in ecology describes positive interactions in which one species enhances another's growth or survival by improving local conditions--through shading, nutrient enrichment, or other stress-reducing mechanisms~\cite{Brooker2008}. Such cooperative interactions are especially important in stressful environments such as drylands~\cite{BertnessCallaway1994}, which collectively cover 41\% of Earth's land surface and support over 38\% of the human population. Under water scarcity, facilitation can allow individuals to enhance each other's growth or survival through mechanisms such as microclimate amelioration, nutrient enrichment, water retention and improved moisture, or physical protection~\cite{Soliveres2015}. Field observations in drylands demonstrate nonlinear responses to aridity, with an abrupt transition from functional to degraded states, consistent with critical transitions in complex ecological systems~\cite{Berdugo2017}. 

As mentioned, both mathematical and computational research have provided results to understanding how habitat loss shapes population and metapopulation dynamics. Despite this progress, most of the classical models omit explicit mechanisms for habitat destruction under facilitation, potentially underestimating the vulnerability of ecological systems where these cooperative feedbacks are relevant. This gap was addressed partially in \cite{JosVidBlaiErnest2021} (see also~\cite{Sardanyes2019}). These authors analysed predator-prey models incorporating a quadratic growth term modelling facilitation in a resource species, also considering a habitat loss parameter directly reducing the carrying capacity. Their analysis revealed how habitat loss reduce equilibrium densities and precipitates discontinuous population collapses.

Building on that work, Seenivasan and Paul in \cite{Seenivasan2025} examined the same facilitation-habitat loss framework in greater detail, including both Holling type I and II functional responses for predation. The inclusion of Holling type II allowed them to capture saturation effects in predator feeding rates, producing qualitatively different stability regimes and earlier predator extinction thresholds compared to the case with no saturation in predation (type I). Their analyses, along with the findings reported in \cite{JosVidBlaiErnest2021}, were limited to local dynamical properties and parameter-space exploration, without extending to a global characterization of phase-space topology or the bifurcation structure governing large-scale system transitions.

Global dynamics analyses in ecological models are essential because they reveal the full range of possible long-term behaviours--such as coexistence, extinction, chaos, or oscillations--rather than just local stability near equilibria. By identifying invariant structures and separatrices in phase space, global analysis clarifies how small parameter changes can redirect ecosystems toward entirely different states. This perspective is crucial for detecting tipping points and understanding abrupt regime shifts driven by processes like habitat loss, exploitation, or climate change, especially when they are governed by global bifurcations.

A broad class of predator-prey systems with a strong Allee effect in the prey was studied in \cite{saleeWang2011}. Under suitable assumptions on the functional response and parameter ranges, the authors established the occurrence of a Hopf bifurcation, the existence of a unique limit cycle, its disappearance through a heteroclinic loop, and the absence of periodic orbits beyond this global bifurcation. These results, however, rely on specific hypotheses and do not provide a description of the phase portraits. The model considered here, given by \eqref{eq:1}, belongs to this general class, which traces back to the pioneering works of Bazykin and Berezovskaya \cite{BazykinBerezovskaya1979} and was later popularized in \cite{1998Bazykin}. In the literature it is commonly referred to as the Bazykin--Berezovskaya model. Our approach differs from that of \cite{saleeWang2011} in that we aim at a full global characterization of the dynamics in the first quadrant, including an explicit description of the bifurcation curves.

In the present work, we extend the results for the resource-consumer model with facilitation and habitat loss provided by Vidiella et al. \cite{JosVidBlaiErnest2021}. By using techniques from the qualitative theory of dynamical systems, we provide a complete description of the global dynamics, classifying all the possible phase portraits in the Poincar\'e disk over the relevant ecological parameter ranges. A central result is the description of the heteroclinic bifurcation curve in the parameter space, which marks the transition between coexistence and full extinction regimes as habitat loss increases. To complement these analyses and gain a comprehensive understanding, we construct a piecewise-linear (PWL) analogue that preserves the qualitative behaviour of the original model across the parameter space. This surrogate system allows an explicit determination of the heteroclinic connection, offering an accessible yet rigorous way to explore the ecological implications of global structural changes in the dynamics. Finally, we numerically explore how extrinsic noise impacts dynamics, including stochasticity in the resource species. These results show that stochastic forcing can advance the onset of collapse, effectively lowering the habitat loss threshold at which extinctions occur. This result highlights the combined destabilizing effects of habitat destruction and environmental variability under facilitation, with direct consequences for predicting and managing ecological resilience, e.g., drylands.

\section{Mathematical model} \label{Sec2}
Here we introduce the mathematical model describing the population dynamics of a  resource species, i.e., vegetation ($V$), and a consumer ($S$) analysed in \cite{JosVidBlaiErnest2021}. The model considers facilitation processes in the resource species, being subjected to habitat loss and consumption by the other species. By assuming a Holling type I functional response, the model reads:
\begin{equation}\label{eq:1}
	\begin{aligned}
		\frac{dV}{dt}&=\alpha \,V^{2}(1-D-V)-\epsilon \,V-\epsilon_{S} \,V \,S,\\ \frac{dS}{dt}&=\mu \,\epsilon_{S}\,V \,S-\delta \,S.
	\end{aligned}
\end{equation}

The parameter $\alpha$ corresponds to the intrinsic growth rate of the vegetation, which is not exponential but hyperbolic due to the process of facilitation, i.e., autocatalysis~\cite{Sardanyes2010,Sardanyes2019}. This growth is limited by a logistic function with normalized carrying capacity that includes the fraction of habitat loss $D \in [0,1)$~\cite{Sardanyes2019,Sole2024}, and by intraspecific competition (e.g., limited resources), whereas $\epsilon$ accounts for its mortality rate. Herbivores $S$ consume vegetation at a rate $\epsilon_S$, with an energetic efficiency parameter $\mu \in [0,1]$, and their population declines due to a natural mortality rate $\delta$. Since the system will be rescaled to facilitate analysis (see below), understanding these parameters is essential for translating the results back to the original formulation.

\begin{pro}\label{pro:equilibirumparameters}
	The original system \eqref{eq:1} is equivalent to the rescaled system
	\begin{equation}\label{eq:2}
		\begin{aligned}
			\dot{x} &= -A x^3 + B x^2 - x y - x, \\
			\dot{y} &= F x y - G y,
		\end{aligned}
	\end{equation}
	under the variable transformations
\begin{equation*}\label{eq:transformation}
	V := l \cdot x, \quad S := n \cdot y, \quad t := \frac{t}{e},
\end{equation*}
and parameter choices  
\begin{equation*}\label{eq:parameter_choices}
	n := \frac{e}{\epsilon_S}, \quad l := 1, \quad e := \epsilon.
\end{equation*}
The new dimensionless parameters are  
\begin{equation}\label{eq:3}
	A = \frac{\alpha}{\epsilon}, \quad B = \frac{\alpha (1 - D)}{\epsilon}, \quad F = \frac{\epsilon_S \mu}{\epsilon}, \quad G = \frac{\delta}{\epsilon}.
\end{equation}
\end{pro}

In Table~\ref{ta:1}, we present the valid ranges of parameters for both the original and the rescaled systems.  
\begin{table}[h]
	\centering
	\begin{minipage}{0.45\textwidth}
		\centering
		\begin{tabular}{|c|}
			\hline
			Parameters of original system \\ \hline
			$\alpha\in [0,\infty)$ \\ \hline
			$D\in[0,1]$ \\ \hline
			$\epsilon\in(0,\infty)$ \\ \hline
			$\epsilon_{S}\in[0,\infty)$ \\ \hline
			$\mu\in[0,1)$\\ \hline
			$\delta\in[0,\infty)$ \\ \hline
		\end{tabular}
	\end{minipage}%
	\hspace{0.05\textwidth}
	\begin{minipage}{0.45\textwidth}
		\centering
		\begin{tabular}{|c|}
			\hline
			Parameters of rescaled system \\ \hline
			$A\in[0,\infty)$ \\ \hline
			$B\in[0,\infty)$ \\ \hline
			$F\in[0,\infty)$ \\ \hline
			$G\in[0,\infty)$ \\ \hline
		\end{tabular}
	\end{minipage}
	\caption{Parameter ranges for the original and the rescaled systems.}\label{ta:1}
\end{table}

\section{Results and discussion}\label{Sec3}

\subsection{Equilibrium points and local dynamics}\label{subsec:EP}

In this section, we determine the equilibrium points of system~\eqref{eq:2}, also introducing a more convenient formulation of the system in which the coordinates of the equilibria are taken as parameters. This representation allows us to study the dynamics in terms of the relative positions of the equilibria. We also analyse the stability of the equilibria by means of linear stability analysis.
\subsubsection{Equilibrium points}

To simplify notation and computations, we assume that the equilibrium points \((x^*, y^*) \in \mathbb{R}^+ \times \mathbb{R}^+\) lie in the first quadrant. We consider only the case in which their number is maximal, that is, when \(B^2 - 4A > 0\). Under this condition, system~\eqref{eq:2} has four equilibria, given by:
\begin{equation*}
	\begin{aligned}
		P^{*}_{0} &= (x_{0}^{*}, y_{0}^{*}) = (0, 0), &P^{*}_{1-} &= (x^{*}_{1-}, y^{*}_1) = \left( \frac{B - \sqrt{B^{2} - 4A}}{2A}, 0 \right),  \\
		P^{*}_{1+} &= (x^{*}_{1+}, y^{*}_1) = \left( \frac{B + \sqrt{B^{2} - 4A}}{2A}, 0 \right), 
		&P^{*}_{2} &= (x_2^{*}, y_2^{*}) = \left( \frac{G}{F}, -\frac{AG^{2} - BGF + F^{2}}{F^{2}} \right). 
	\end{aligned}
\end{equation*}

\subsubsection{Redefining parameters}

To further simplify the study of system \eqref{eq:1}, the parameters of \eqref{eq:2} can be chosen so that the relative position of the equilibrium points determines their qualitative behaviour. 

\begin{pro}\label{l1} 
	Let $x_{0}, x_{1}, x_{e}, F > 0$ and assume that
	$$
	A = \frac{1}{x_{0} x_{1}}, \quad B = \frac{x_{0} + x_{1}}{x_{0} x_{1}}, \quad G = x_{e} F.
	$$
	Then, the differential system \eqref{eq:3} can be written as
	\begin{equation}\label{eq:4}
		\begin{aligned}
			\dot{x} &= -\frac{x^3}{x_{0} x_{1}} + \frac{(x_{0} + x_{1}) x^2}{x_{0} x_{1}} - x y - x, \\
			\dot{y} &= F (x y - x_{e} y)
		\end{aligned}
	\end{equation}
	and its equilibrium points are
	$$
	P^{*}_{0} = (0, 0), \; P^{*}_{1-} = (x_{0}, 0), \; P^{*}_{1+} = (x_{1}, 0) \text{, and } P^{*}_{2} = \left(x_{e}, \frac{(x_{1} - x_{e})(x_{e} - x_{0})}{x_{0} x_{1}}\right) =: (x_{e}, y_e).
	$$
\end{pro}
Here, we provide the explicit correspondence between the new parameters $(x_0, x_1, x_e, F)$ and the original parameters $(\alpha, D, \epsilon, \epsilon_S, \mu, \delta)$, which simplifies the ecological interpretation of the dynamical results.

\begin{equation}\label{eq:equilibria-original-parameters}
x_{0,1}
=\frac{1-D \mp \sqrt{(1-D)^2 - 4\varepsilon/\alpha}}{2},
\qquad
x_{e} = \frac{\delta}{\varepsilon_S \mu},\qquad
F = \frac{\varepsilon_S \mu}{\varepsilon}.
\end{equation}

Their ecological roles are as follows:
\begin{enumerate}[(i)]
    \item $x_0$:
    The lower vegetation equilibrium in the absence of consumers, representing the unstable threshold below which vegetation cannot recover. Since $x_0$ increases with habitat destruction, vegetation becomes unstable at increasingly higher biomass levels, reflecting a loss of resilience.
    \item $x_1$:
    The upper vegetation equilibrium (carrying capacity) without consumers.  The greater habitat destruction reduces $x_1$ and thus lowers the maximum sustainable vegetation biomass.
    \item  $x_e$:
    The ratio between consumer mortality and effective consumption. 
	\item  $F$:  
    The ratio between effective consumer feeding and vegetation mortality.  
    Small values of $F$ arise when vegetation mortality dominates, which generates the sharp slow-fast dynamics discussed in Section~\ref{subsec:slowfast}. Conversely, large values of $F$ correspond to vegetation that is highly resilient to consumption.
\end{enumerate}

\subsubsection{Configuration of equilibria in the first quadrant}
Once equilibria of \eqref{eq:4} have been determined, we analyse the most relevant configurations from an ecological perspective, focusing on those in which all four equilibria lie in the first quadrant. These possible configurations of the four equilibrium points satisfying (${0\leq x_{0}\leq x_{e}\leq x_{1}}$), are:
\begin{enumerate}[(i)]
    \item Transcritical bifurcation: $0<x_{0}=x_{e}<x_{1}$ and $y_{e}=0$.
    \item Generic configuration: $0<x_{0}<x_{e}<x_{1}$ and $y_{e}>0$.
    \item Transcritical bifurcation: $0<x_{0}<x_{e}=x_{1}$ and $y_{e}=0$.
    \item Saddle-node bifurcation: $0<x_{0}=x_{e}=x_{1}$ and $y_{e}=0$.
    \item Transcritical and saddle-node bifurcation: $0=x_{0}=x_{e}=x_{1}$ and $y_{e}=0$. 
\end{enumerate}
	
In what follows, we restrict our analysis to the generic configuration ($0 < x_{0} < x_{e} < x_{1}$, $y_{e} > 0$), corresponding to feasible coexistence of vegetation and consumers. This is the ecologically relevant regime and the one in which the system exhibits its most complex dynamics.

\subsubsection{Local analysis of equilibria}
To study the behaviour of the equilibrium points, we linearize the system \eqref{eq:4} and apply the Hartman--Grobman Theorem to analyse the local dynamics around these points. The Jacobian matrix of the system at an equilibrium point \((x^*, y^*)\) is given by

\begin{equation*}
	J(x^{*}, y^{*}) =
	\begin{pmatrix}
		-\dfrac{3 (x^{*})^{2}}{x_{1} x_{0}}
		+ \dfrac{2 (x_{0} + x_{1}) x^{*}}{x_{1} x_{0}}
		- y^{*} - 1
		& -x^{*} \\[10pt]
		F y^{*} & F (x^{*} - x_{e})
	\end{pmatrix}.
\end{equation*}

In a planar system, the dynamical behaviour of equilibria are characterized by the trace ($\tau$) and the determinant ($\Delta$) of the linearized system, so these will be calculated for each equilibrium:
\begin{equation}\label{eq:tracedet}
	\begin{aligned}
		P^{*}_{0}: \quad & \tau(0,0) = -x_{e}F - 1, \quad & \Delta(0,0) &= x_{e}F, \\
		P^{*}_{1-}: \quad & \tau(x_{0},0) = \frac{F x_{1}(x_{0} - x_{e}) + (x_{1} - x_{0})}{x_{1}}, \quad & \Delta(x_{0},0) &= -\frac{F(x_{1} - x_{0})(x_{e} - x_{0})}{x_{1}}, \\
		P^{*}_{1+}: \quad & \tau(x_{1},0) = \frac{F x_{0}(x_{1} - x_{e}) + (x_{0} - x_{1})}{x_{0}}, \quad & \Delta(x_{1},0) &= -\frac{F( x_{1}-x_{0})(x_{1}-x_{e})}{x_{0}}, \\
		P^{*}_{2}: \quad & \tau(x_{e},y_{e}) = \frac{x_{e}(-2x_{e} + x_{0} + x_{1})}{x_{1} x_{0}}, \quad & \Delta(x_{e},y_{e}) &= \frac{F x_{e}(x_{1}-x_{e})(x_{e} - x_{0})}{x_{1} x_{0}}.
	\end{aligned}
\end{equation}

\begin{pro}
For $0<x_{0}<x_{e}<x_{1}$, equilibrium points have the following behaviour:
	\begin{enumerate}[(i)]
		\item $P^{*}_{0}$ is a stable node. 
		\item $P^{*}_{1-}$ is a saddle point. 
        \begin{enumerate}[(a)]
			\item Eigenvalues: \( \lambda_{+}^{0} = \frac{x_{1} - x_{0}}{x_{1}} > 0 \) and \( \lambda_{-}^{0} = F(x_{0} - x_{e}) < 0 \).
			\item Eigenvectors: \( v_{+}^{0} =(1,0)\) and \( v_{-}^{0} = \left(\frac{x_{0} x_{1}}{F x_{1}(x_{e} - x_{0}) + x_{1} - x_{0}} , 1 \right) \).
		\end{enumerate}
		\item  $P^{*}_{1+}$ is a saddle point. 
		\begin{enumerate}[(a)]
			\item Eigenvalues: \( \lambda_{+}^{1} = F(x_{1} - x_{e}) > 0 \) and \( \lambda_{-}^{1} = \frac{x_{0} - x_{1}}{x_{0}} < 0 \).
			\item Eigenvectors: \( v_{+}^{1} = \left( -\frac{x_{0} x_{1}}{F x_{0}(x_{1} - x_{e}) + x_{1} - x_{0}} , 1 \right) \) and \( v_{-}^{1} = \left( 1 , 0 \right) \).
		\end{enumerate}
		\item $P^{*}_{2}$ can be either a stable node, a stable focus, an unstable node, or an unstable focus, depending on the parameter conditions given in Table~\ref{tab:stability_cases}.
	\end{enumerate}
		
	\begin{table}[H]
	\centering
	\renewcommand{\arraystretch}{2.2}
	\begin{tabular}{|c|c|}
		\hline
		\textbf{Parameter Conditions} & \textbf{Type} \\ 
		\hline
		$\dfrac{x_{0}+x_{1}}{2}>x_{e}$, 
		$\dfrac{x_{e}(-2x_{e} + x_{0} + x_{1})^{2}}
		{4 x_{1} x_{0} (x_{1} - x_{e})(x_{e} - x_{0})}\geq F$ 
		& Unstable Node \\[6pt]
		\hline
		$\dfrac{x_{0}+x_{1}}{2}>x_{e}$, 
		$\dfrac{x_{e}(-2x_{e} + x_{0} + x_{1})^{2}}
		{4 x_{1} x_{0} (x_{1} - x_{e})(x_{e} - x_{0})}<F$ 
		& Unstable Focus \\[6pt]
		\hline
		$\dfrac{x_{0}+x_{1}}{2}<x_{e}$, 
		$\dfrac{x_{e}(-2x_{e} + x_{0} + x_{1})^{2}}
		{4 x_{1} x_{0} (x_{1} - x_{e})(x_{e} - x_{0})}\geq F$ 
		& Stable Node \\[6pt]
		\hline
		$\dfrac{x_{0}+x_{1}}{2}<x_{e}$, 
		$\dfrac{x_{e}(-2x_{e} + x_{0} + x_{1})^{2}}
		{4 x_{1} x_{0} (x_{1} - x_{e})(x_{e} - x_{0})}<F$ 
		& Stable Focus \\[6pt]
		\hline
	\end{tabular}
	\caption{Classification of the equilibrium point $P^{*}_{2}$ based on parameter conditions.}	\label{tab:stability_cases}
        \end{table}

	\begin{proof}
		From the values obtained in \eqref{eq:tracedet} we can describe the local equilibrium behaviour.
		\begin{enumerate}[(i)]
			\item $\Delta(0,0)>0$, $\tau(0,0)<0$ and $\tau(0,0)^{2}-4\Delta(0,0)=(x_{e}^{2}F^{2}-1)^{2}>0$. Hence, $P^{*}_{0}$ is a stable node.
			\item $\Delta(x_{0},0)<0$. Hence, $P^{*}_{1-}$ is a saddle.
			\item $\Delta(x_{1},0)<0$. Hence, $P^{*}_{1+}$ is a saddle.
			\item $\Delta(x_{1},0)<0$. The stability of the equilibrium points is determined by the sign of $\tau(x_e, y_e)$, while its nature (whether it is a focus or a node) depends on the expression $\tau(x_e, y_e)^2 - 4\Delta(x_e, y_e)$. Specifically:  
			\begin{enumerate}[(a)]
				\item If $\tau(x_e, y_e) < 0$, the equilibrium is stable; if $\tau(x_e, y_e) > 0$, it is unstable. In Section~\ref{sec:Hopf}, is dealt with the case where $\tau(x_e, y_e)= 0$. 
				\item If $\tau(x_e, y_e)^2 - 4\Delta(x_e, y_e) < 0$, the equilibrium is a focus; if $\tau(x_e, y_e)^2 - 4\Delta(x_e, y_e) \geq 0$, it is a node.  
			\end{enumerate}
		\end{enumerate}
	\end{proof}
\end{pro}
\begin{remark}
	It is well known that, in contrast to quadratic planar differential systems, where a focus is necessary for the existence of a limit cycle, cubic planar differential systems may admit limit cycles surrounding a node (see, for instance, \cite{COPPEL1966293}). Therefore, the fact of having pure real eigenvalues does not influence the topological behaviour of the phase portrait. However, from an ecological standpoint, the presence of population oscillations could have important implications for conservation strategies. The critical curve separating the node and the focus region for the coexistence equilibrium $(x_{e},y_{e})$ is $F_{FN}(x_{e})=\frac{x_{e}(-2x_{e} + x_{0} + x_{1})^{2}}{4 x_{1} x_{0} (x_{1} - x_{e})(x_{e} - x_{0})}$.
\end{remark}

\subsection{Positivity and boundedness of solutions}
All the considered systems belong to the class of Kolmogorov models, originally introduced in \cite{Sigmund2007}, in which each equation can be written in the form
    \begin{equation*}
        \begin{aligned}
            \dot x &= x\,f(x,y),\\
            \dot y &= y\,g(x,y),
        \end{aligned}
    \end{equation*}
with $f$ and $g$ usually smooth functions. In particular, positivity of solutions follows immediately from the invariance of the axes. In this subsection we establish that, in the ecologically relevant parameter range $x_0<x_e<x_1$, all trajectories are uniformly bounded by means of a comparison argument. 

\begin{lema}[Uniform boundedness of solution]
Assume $0<x_0<x_e<x_1$, $0<F<\infty$ and let $(x(t),y(t))$ be a solution of system~\eqref{eq:4} with initial condition
$(x(0),y(0))\in\mathbb{R}^2_+$. Then all trajectories are uniformly bounded. More precisely, there exists $M>0$ such that
$$
0\le x(t)\le M,\quad 0\le y(t)\le M,\quad \text{for all } t\ge 0.$$
Moreover, the set $\mathcal{B}:=\{(x,y)\in\mathbb{R}^2_+:\; x+\tfrac{1}{F}y\le M\}$ is absorbing.
\end{lema}

\begin{proof}
Since $x(t),y(t)\ge0$, we have
\[
\dot x \le -\frac{x}{x_0x_1}(x-x_0)(x-x_1).
\]
Hence, if $x>x_1$ then $\dot x<0$, and therefore
\[
0\le x(t)\le M_x:=\max\{x(0),x_1\}, \qquad t\ge0.
\]

Define $w(t)=x(t)+\tfrac1F y(t)$. Using system~\eqref{eq:4} and $y=F(w-x)$ we obtain
\[
\dot w=\phi(x)-F x_e w,
\]
where
\[
\phi(x)=-\frac{x^3}{x_0x_1}+\frac{x_0+x_1}{x_0x_1}x^2-x+F x_e x.
\]
Since $x(t)\in[0,M_x]$ and $\phi$ is continuous, there exists
\[
C:=\max_{0\le x\le M_x}\phi(x)<\infty
\]
such that $\dot w\le C-Fx_e w$. Solving this inequality yields
\[
w(t)\le \frac{C}{F x_e}+\Big(w(0)-\frac{C}{F x_e}\Big)\ee^{-F x_e t}, \qquad t\ge0.
\]
Consequently, $w$ is uniformly bounded. Since $x\le w$ and $y=F(w-x)\le Fw$, both components are uniformly bounded. Moreover,
\[
\mathcal{B}=\{(x,y)\in\mathbb{R}^2_+:\; x+\tfrac1F y\le M\}
\]
is an absorbing set, which completes the proof.
\end{proof}

\begin{remark}
The existence of the absorbing set $\mathcal{B}$
implies that the semiflow generated by system~\eqref{eq:4} is \emph{dissipative} in $\mathbb{R}^2_+$. In particular, all solutions are uniformly bounded forward in time and every bounded set of initial data is eventually absorbed into the fixed bounded set $\mathcal{B}$.
\end{remark}

\subsection{Uniqueness and characterization of the limit cycle}\label{subsec:ULC}

Since the equilibrium points $P_{0}^{*}, P_{1-}^{*}$, and $P_{1+}^{*}$ lie on the invariant line $y = 0$, no periodic orbit can exist around them. Thus, the only equilibrium that could have a surrounding limit cycle is $P_{2}^{*}$.
\begin{tho} (Uniqueness of limit cycles for generalized Li\'enard systems \cite{Zegeling})\label{th:uniqueness}
	
	Consider the system of differential equations 
	\[
	\frac{dx}{dt} = F_{L}(x) - \psi(y), \quad \frac{dy}{dt} = g_{L}(x).
	\]
	If the following conditions are satisfied in the band $x_{-} < x < x_{+}$:
	\begin{enumerate}[(i)]
		\item $(x - x_{g_{L}})g_{L}(x) > 0,$ for $x \neq x_{g_{L}}$;
		\item $(x - x_{f_{L}})f_{L}(x) > 0,$ for $x \neq x_{f_{L}}$, $x_{f_{L}} > x_{g_{L}}$;
		\item $\frac{d\psi(y)}{dy} > 0$;
		\item $\frac{d}{dx}\left(\frac{f_{L}(x)}{g_{L}(x)}\right) \leq 0$ in $x_{-} < x < x_{g_{L}}$ and $x_{f_{L}} < x < x_{+}$;
	\end{enumerate}
	then the system has at most one unique stable and hyperbolic limit cycle in the band $x_{-} < x < x_{+}$.
\end{tho}	

\begin{coro} \label{coro:uniqueness}
	The system of differential equations  
	\begin{equation*}
		\dot{x} = -Ax^{2} + Bx - 1 - \ee^{v}, \quad \dot{v} = \frac{Fx - G}{x},
	\end{equation*}
	has at most one unique stable and hyperbolic limit cycle in the strip $ x_{0} < x < x_{1} $.
	
	\begin{proof}
		By performing a change of coordinates in system~\eqref{eq:2} given by $y=\ee^{v}$ and rescaling time as $t \rightarrow \frac{t}{x}$, the system can be rewritten as a generalized Li\'enard system: 
		\begin{equation*}
			\dot{x} = -Ax^{2} + Bx - 1 - \ee^{v} = F_{L}(x) - \psi(y), \quad \dot{v} = \frac{Fx - G}{x} = g_{L}(x).
		\end{equation*}
		By following Theorem~\ref{th:uniqueness}:
		
		\begin{enumerate}[(i)]
			\item $(x - x_{g_{L}}) g_{L}(x) = (x - \frac{G}{F}) \cdot \frac{Fx - G}{x} = \frac{(Fx - G)^{2}}{Fx} > 0$, for $ x \neq x_{g_{L}} $ and $ x > 0 $;
			\item $(x - x_{f_{L}}) f_{L}(x) = (x - \frac{B}{2A}) \cdot (-2Ax + B) = \frac{(2Ax - B)^{2}}{2A} > 0$, for $ x \neq x_{f_{L}} $, with $ x_{f_{L}}> x_{g_{L}} $;
			\item $ \frac{d\psi(y)}{dy} > 0 $;
			\item $ \frac{d}{dx} \left( \frac{f_{L}(x)}{g_{L}(x)} \right) = -\frac{2AFx^{2} - 4AGx + GB}{(Fx - G)^{2}} \leq 0 $ in $ x_{-} < x < x_{g_{L}} $ and $ x_{f_{L}}< x < x_{+} $.
		\end{enumerate}
		
		Thus, there is at most one stable and hyperbolic limit cycle in the strip  $0 < x < \infty$. Moreover, the stable limit cycle must be within $ x_{0} < x < x_{1}$, as it is the only feasible candidate in the strip $0 < x < \infty$. To bound the region where a limit cycle may appear, we analyse the vector field along the lines $x=x_0$ and $ x=x_1$:  
		\begin{enumerate}[(a)]
			\item At \( x = x_0 \):  
			\[
			\dot{x} = -x_0 y < 0, \quad \dot{y} = F y (x_0 - x_e) < 0
			\]
			\item At \( x = x_1 \):  
			\[
			\dot{x} = -y x_1 < 0, \quad \dot{y} = F y (x_1 - x_e) > 0
			\]
		\end{enumerate}
		Since the vector field crosses these lines in only one direction, any limit cycle must be entirely contained within the region $x_0 < x < x_1$, without crossing these boundaries.
		\end{proof}
\end{coro}
Moreover, the limit cycle can only exist in some region of the parameter space.
 \begin{tho}\label{th:BD}
 	For $\frac{x_{0}+x_{1}}{2}\leq x_e$,  system \eqref{eq:4} has no limit cycle.
 	\begin{proof}
    Given the system \eqref{eq:4} and the Dulac function $B(x, y) = x^a y^b$, we want the expression

$$
M := \frac{\partial(\dot{x} \cdot B)}{\partial x} + \frac{\partial(\dot{y} \cdot B)}{\partial y}
$$
to preserve sign. If this holds, it proves the nonexistence of periodic orbits within some subregion of the parameter space. Assuming $a = -1$, we obtain

$$
M := -2x^2 + \left(b F x_1 x_0 + F x_1 x_0 + x_0 + x_1\right) x - b F x_1 x_0 x_e - F x_1 x_0 x_e.
$$

In the parameter space $(x_e, f)$, the zeros of $M(x)$ are curves in the parameter space given by
\begin{eqnarray} 
x^{\ast} &=& \frac{(1 + F(b + 1)x_1)x_0+x_1}{4} \nonumber\\ &\pm& \frac{\sqrt{(1 + F(b + 1)x_1)^2 x_0^2 + 2(F(b + 1)x_1 + 1 - 4x_e(b + 1)F)x_1 x_0 + x_1^2}}{4}.\nonumber
\end{eqnarray}
To avoid sign changes in $M$, we want this expression to be complex-valued, i.e., the discriminant is imposed to be negative. The boundary where it vanishes defines a critical curve 

$$
F(x_e) = \frac{-x_0 - x_1 + 4x_e \pm 2 \sqrt{-2x_0 x_e - 2x_1 x_e + 4x_e^2}}{(b + 1)x_1 x_0},
$$
in the parameter space $(x_{e},F)$. The parabola controls the sign change of the divergence.
Since $F \in \mathbb{R}^+$, this implies $x_e \geq \frac{x_0 + x_1}{2}$. Qualitatively, the parabola behaves as shown in Figure~\ref{fi:1}.

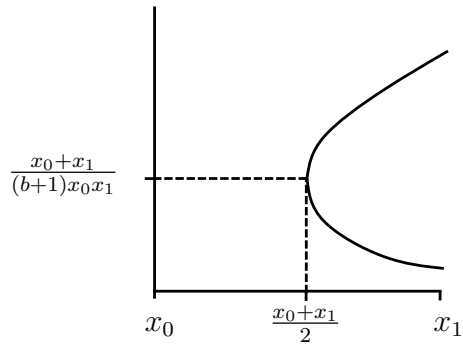
\begin{figure}[h]
	\centering
	\begin{overpic}[width=0.25\textwidth]{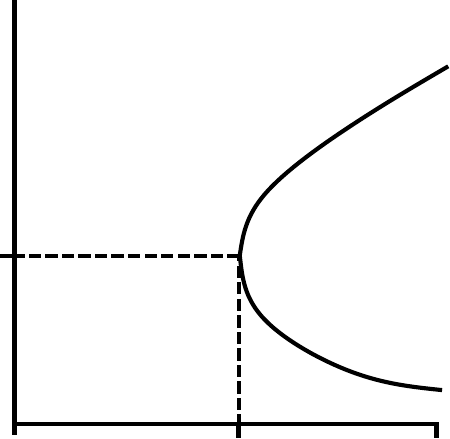}
		\put(-0,-10){$x_{0}$}
		\put(41,-10){$\frac{x_{0}+x_{1}}{2}$}
		\put(95,-10){$x_{1}$}
		\put(-45,40){$\frac{x_{0}+x_{1}}{(b+1)x_{0}x_{1}}$}
	\end{overpic}
	\vspace{0.3cm}
    	 \captionsetup{width=\textwidth}
	\caption{Critical curve $f(x_e)$ indicating change in divergence sign.}\label{fi:1}
\end{figure}

Since $b$ remains as a free parameter, we can choose arbitrary values of $b$ to cover certain regions of the parameter space. If we take values of $b \in (-1, \infty)$, then for each fixed $b$, the interior of the corresponding curve, denoted $R_b$, satisfies

$$
\bigcup R_b = \left\{(x_e, F) \in \mathbb{R}^2 \,\middle|\, \frac{x_0 + x_1}{2} < x_e \leq x_1 \text{ and } F > 0 \right\}.
$$

Also, for $x_e=\frac{x_0+x_1}{2}$ the unique limit cycle emerging at the Andronov--Hopf bifurcation has not emerged yet (see details in Section~\ref{sec:Hopf}). Therefore, no limit cycles can exist for $\frac{x_0 + x_1}{2} \leq x_e \leq x_1$
 	\end{proof}
 \end{tho}

Further on, we show that the system exhibits both a Hopf bifurcation (Section~\ref{sec:Hopf}) and a heteroclinic bifurcation (Section~\ref{sec:heteroclinic}), which together determine the parameter region where the unique limit cycle can exist. In particular, Corollary~\ref{coro:uniqueness} shows that the region of parameter space in which there is no limit cycle is larger than previously established in Theorem~\ref{th:BD}.

The Hopf bifurcation governs the local emergence or disappearance of small-amplitude periodic solutions near the coexistence equilibrium, whereas the heteroclinic bifurcation specifies the parameter values at which the limit cycle and the polycycle coincide. Furthermore, Proposition~\ref{pro:dieLC} shows that the limit cycle created at the Hopf bifurcation necessarily vanishes at the heteroclinic connection, as no alternative mechanism for its destruction exists in this parameter range. This result does not ensure that the limit cycle grows monotonically along the parameter path linking the two bifurcations, the fact that it terminates on a global structure such as the polycycle--and that its period diverges when approaching the heteroclinic connection--indicates that both the amplitude and the period of the oscillations carry valuable information about the collapse.

Although the following result has no direct ecological implications, it provides a useful mathematical characterization of the invariant curves of the system. Algebraic invariant curves are typically easier to detect and analyse and--most importantly--they can be written explicitly, which greatly facilitates numerical continuation. They also play a central role in the construction of first integrals, since algebraic invariant curves frequently serve as building blocks for Darboux-type integrability (see Chapter 8 in \cite{DumLliArt2006}). Moreover, the theorem below shows that phase portraits admitting algebraic invariant curves other than the coordinate axes form a zero-measure subset of the parameter space; hence, for generic parameter values, no such algebraic curves exist.
\begin{tho}[\cite{gasull2014non}]\label{th:trascendentallc}
Consider the planar polynomial differential system
\begin{equation}\label{eq:tracendetallc}
    \begin{cases}
    \dot{x} = x(a_0 + a_1 y + a_2 y^2), \\
    \dot{y} = y(x + b_0 + b_1 y + b_2 y^2 + b_3 y^3),
    \end{cases}
\end{equation}
where $a_0 \neq 0$ and $\dfrac{b_0}{a_0} \notin \mathbb{Q}^+$. Then the only invariant algebraic curves of the system are the coordinate axes $x = 0$ and $y = 0$. In particular, all limit cycles of the system are non-algebraic.
\end{tho}

\begin{coro}
Let $\frac{1 - B}{G} \notin \mathbb{Q}^+, \quad \text{and} \quad G \neq 0.$
Then, there exists a change of variables and a time reversal such that the system \eqref{eq:2} transforms into system~\eqref{eq:tracendetallc}. Consequently, the only invariant algebraic curves are $x = 0$ and $y = 0$, and the only limit cycle is transcendental.
\end{coro}

\begin{proof}
Applying the change $x=\bar y,\; y=\bar x$ and the time reversal $s=-t$,
system~\eqref{eq:2} becomes
$$
\dot{\bar x}=\bar x(G-F\bar y),\qquad
\dot{\bar y}=\bar y(\bar x+1-B\bar y+A\bar y^2).
$$
Renaming variables gives a system of the form \eqref{eq:tracendetallc} with
$a_0=G\neq 0$ and ${b_0/a_0=1/G\notin\mathbb{Q}^+}$.  
Hence Theorem~\ref{th:trascendentallc} applies: the only invariant algebraic
curves are $x=0$ and $y=0$, and the unique limit cycle is transcendental.
\end{proof}

\subsection{Rotated vector field}\label{subsec:rvfmodel}

Using the theory of rotated vector fields detailed in Appendix~\ref{sec:rvf}, we can ensure the monotonic movement of the invariant manifolds with respect to $x_{e}$. In general, this framework is useful for tracking the evolution of objects, such as limit cycles, within a vector field as a parameter varies. In our case, it serves to characterize the heteroclinic bifurcation curve and conclude that it is expressible as $x_{e}(F)$.

By inspection of the nullclines, the heteroclinic connection between 
$W^u(x_0,0)$ and $W^s(x_1,0)$ lies within the region
$$
\Omega = \{(x,y)\in(\mathbb{R}^{2}_{+}\mid 
y>-\tfrac{(x-x_{0})(x-x_{1})}{x_{0}x_{1}}\}.
$$
Hence, we can restrict the application of the rotated vector fields theory to this region. 

By computing the determinant
\begin{equation*}
    \Theta_{x_e}(x,y)
=
\begin{vmatrix}
P(x,y) & Q(x,y)\\[2pt]
\partial_{x_e}P(x,y) & \partial_{x_e}Q(x,y)
\end{vmatrix}
= -F y\,P(x,y)
> 0
\quad \text{for } (x,y)\in\Omega,
\end{equation*}
we see that the sign is preserved in $\Omega$. Hence, as $x_e$ increases, the vector field rotates monotonically counter-clockwise in $\Omega$, and the rotated vector fields theory applies in $\Omega$ with $x_e$ as the rotatory parameter.

\begin{remark}
    As detailed in Appendix~\ref{sec:rvf}, the rotated vector field theory applies to one-parameter families of vector fields in which the equilibria remain fixed as the parameter varies. 
    In our case, the rotatory parameter is $x_e$, so the coexistence equilibrium $(x_{e}, y_{e})$ moves when $x_e$ changes. 
    However, this displacement occurs outside the region $\Omega$, and therefore the rotated vector field theory remains applicable.
\end{remark}

\begin{pro}\label{pro:xe_f}
The heteroclinic bifurcation curve in $(x_{e},F)$-space is the graph of a
function $x_e^h(F)$.
\end{pro}

\begin{proof}
The constant sign of $\Theta_{x_e}$ in $\Omega$ implies that the invariant
manifolds vary monotonically with $x_e$, so for each $F$, there is at most one
$x_e$ producing the heteroclinic connection; hence, the curve is a graph.
\end{proof}

\subsection{Global behaviour of the phase portrait}
To identify the limit sets of the phase portrait of system~\eqref{eq:4}, we use
the Poincar\'e compactification (see Appendix~\ref{sec:PC}). This procedure
embeds any planar polynomial differential system into the Poincar\'e sphere,
turning the behaviour at infinity into dynamics on its equator and yielding a
compact phase space. In this setting the Poincar\'e--Bendixson Theorem applies,
allowing a complete description of the possible limit sets. The analysis is
performed through the local charts that cover the sphere, in which the vector
fields take the following form:

\begin{enumerate}[(a)]
	\item Local chart $(U_1,\phi_1):$
	\begin{equation*}
		\begin{aligned}
			\dot{u} &= -\frac{u\left(v\left(1+\left(F v x_e - F - u - v\right)x_1\right)x_0 + x_1 v - 1\right)}{x_0 x_1},\\
			\dot{v} &= \frac{(1 + v^{2} x_0 x_1 + (u x_0 x_1 - x_0 - x_1)v)\,v}{x_0 x_1}.
		\end{aligned}
	\end{equation*}

	\item Local chart $(U_2,\phi_2):$
	\begin{equation*}
		\begin{aligned}
			\dot{u} &= -\frac{u\big(v\big((1+v+F u - F v x_e)x_1 - u\big)x_0 + u(u - x_1 v)\big)}{x_0 x_1},\\
			\dot{v} &= -v^{2} F(-x_e v + u).
		\end{aligned}
	\end{equation*}

	\item Local chart $(U_3,\phi_3):$
	\begin{equation*}
		\begin{aligned}
			\dot{u} &= -\frac{u^{3}}{x_0 x_1} + \frac{(x_0 + x_1)u^{2}}{x_0 x_1} - u v - u,\\
			\dot{v} &= F u v - F x_e v.
		\end{aligned}
	\end{equation*}
\end{enumerate}

Chart $(U_1,\phi_1)$ blows up the direction $(\infty,0)$ to the origin, and
chart $(U_2,\phi_2)$ does the same for the direction $(0,\infty)$. For a detailed account of the analysis of singular points at infinity,
\cite{DizPitaBlowUpsInfinity} is a useful reference, particularly in
the context of predator-prey systems.

\begin{pro}
The origin of chart $(U_1,\phi_1)$ is a repellor.
\end{pro}

\begin{proof}
The Jacobian matrix at $(u,v)=(0,0)$ is
$$
J(0,0)=
\begin{pmatrix}
\frac{1}{x_0 x_1} & 0\\[3pt]
0 & \frac{1}{x_0 x_1}
\end{pmatrix},
$$
which is hyperbolic with both eigenvalues positive.  
By the Hartman--Grobman Theorem the origin is a repellor.
\end{proof}

\begin{pro}
The origin of chart $(U_2,\phi_2)$ is a saddle-node.
\end{pro}

\begin{proof}
The Jacobian matrix at $(u,v)=(0,0)$ is
$$
J(0,0)=
\begin{pmatrix}
0 & 0\\[3pt]
0 & 0
\end{pmatrix},
$$
so the equilibrium is non-elementary.  
To desingularize it, apply the directional blow-up $u = v w$, obtaining
$$
\begin{aligned}
\dot{w} &=-\frac{\big((v+1)x_1 - v w\big)x_0 + v w (w - x_1)}{x_0 x_1}\, w
,\\[4pt]
\dot{v} &= -v^{2} F (w-x_e).
\end{aligned}
$$
This non-elementary equilibrium is a saddle-node, with a hyperbolic sector in the first quadrant of the compactified vector field. This can be seen using classical classification of equilibria results. See, for example, Theorem~2.19 in \cite{DumLliArt2006}.
\end{proof}

\subsection{Local bifurcations}
In the model we encounter the following local bifurcations:
\begin{enumerate}[(a)]
	\item Saddle-node at $x_{0}=x_1=x_{e}$.
	\item Transcritical at $x_{e}=x_1$ and $x_{e}=x_0$.
	\item Andronov--Hopf at $x_{e}=\frac{x_{0}+x_{1}}{2}$.
\end{enumerate}
\subsubsection{Saddle-node bifurcation}\label{subsubsec:saddlenode}

The saddle-node bifurcation corresponds to the collision and annihilation of two
equilibria with different stability. In system~\eqref{eq:4}, this occurs along the axis
$y=0$, where the dynamics reduce to
$$
\dot{x}
= -\frac{x^3}{x_0x_1}
  +\frac{(x_0+x_1)x^2}{x_0x_1}
  - x,
\qquad
\dot{y}=0.
$$
The equilibrium points on this axis are $(x_0,0)$ and $(x_1,0)$, which always remain
in $y=0$. Their stability follows from
$$
P'(x_0)=1-\frac{x_0}{x_1}>0,
\qquad
P'(x_1)=1-\frac{x_1}{x_0}<0,
$$
so $x_0$ is unstable and $x_1$ is stable whenever $x_0<x_1$.

\medskip
In terms of the original parameters $(\alpha,\varepsilon,D)$, the vegetation
equilibria of system~\eqref{eq:3} are
$$
x_{0,1}
=\frac{1-D \mp \sqrt{(1-D)^2 - 4\varepsilon/\alpha}}{2},
$$
as given in~\eqref{eq:equilibria-original-parameters}.  
They are real and distinct precisely when
$$
(1-D)^2 - \frac{4\varepsilon}{\alpha} > 0
\quad\text{ if and only if }\quad
0 < D < 1 - 2\sqrt{\varepsilon/\alpha},
$$
which defines the \emph{generic regime} where the system possesses the coexistence equilibria.

\subsubsection{Andronov--Hopf bifurcation}\label{sec:Hopf}
 An Andronov--Hopf bifurcation consists in a change of stability of an equilibrium point, induced by a pair of complex conjugate eigenvalues crossing the imaginary axis when at least one small-amplitude limit cycle appears.

\begin{tho}\label{hopf} 
	The equilibrium $P_{2}^{*}$ of system \eqref{eq:4} undergoes a supercritical Hopf bifurcation at $x_{e}=\frac{x_{0}+x_{1}}{2}$. For $x_{e}<\frac{x_{0}+x_{1}}{2}$ it unfolds a limit cycle of small amplitude from $P_{2}^{*}$. Moreover, at $x_{e}=\frac{x_{0}+x_{1}}{2}$ there is a weak focus of order one. 
	\begin{proof}
		To establish the Hopf bifurcation, we verify the following conditions following \cite{kuznetsov2008elements}:  

First the vanishing trace condition: The trace of the linearized system at $P_{2}^{*}$ is
			$$\tau(x_{e},y_{e})=\frac{x_e (x_0 + x_1 - 2x_e)}{x_1 x_0} \text{ and vanishes at } x_{e}=\frac{x_{0}+x_{1}}{2}.$$
			
Second, the Transversality condition, the derivative of the eigenvalues with respect to $x_e$, assuming the vanishing trace condition, is nonzero. This assumption ensures a non-degenerate bifurcation:
			$$\frac{d\,\tau}{d\,x_{e}}\Bigg|_{x_{e}=\frac{x_{0}+x_{1}}{2}} \neq 0.$$
			
Third, the non-zero first Lyapunov constant. The computation yields 
			$$L_{1}=-\frac{F(x_{1}-x_{0})^{2}}{6x_{0}^{2}x_{1}^{2}}<0,$$  
			which confirms that the obtained limit cycle will be stable.  
		
		Thus, at $x_{e}=\frac{x_{0}+x_{1}}{2}$, the equilibrium is a weak focus of order one, and for $x_{e}<\frac{x_{0}+x_{1}}{2}$, a stable limit cycle emerges, consistent with Theorem~\ref{th:uniqueness}.  
	\end{proof}
\end{tho}

We consider the new parameter $\lambda = x_{e} - \frac{x_{0} + x_{1}}{2}$,
so that the Andronov--Hopf bifurcation occurs at $\lambda = 0$. For
convenience, the Hopf condition will be written in terms of $\lambda$
throughout the text, especially in Section~\ref{subsec:slowfast}.

\begin{pro}
	The period of the limit cycle emerging at the Andronov--Hopf bifurcation is
	\begin{equation*}
		T(\lambda,F)=\frac{4\pi\sqrt{2(x_0+x_1)x_0 x_{1}F(x_{0}-x_{1})^{2}}}{F(x_{0}+x_{1})(x_{0}-x_{1})^{2}}\left(1-\frac{1}{(x_0+x_1)}\lambda+\mathcal{O}(\lambda^{2})\right).
	\end{equation*}
\end{pro}

\subsubsection{Transcritical bifurcation}
A transcritical bifurcation takes place when two equilibrium branches
intersect and exchange their stability as a parameter varies.
\begin{pro}
System~\eqref{eq:4} undergoes transcritical bifurcations at the points $x_e = x_0$ and $x_e = x_1$. 
The exchange of stability between the corresponding equilibria is summarized in Table~\ref{tab:34}, and illustrated in Figure~\ref{fig:transcritical}.

\vspace{0.5cm}
\begin{table}[h]
        \begin{tabular}{|c|c|c|}
        \hline
        Condition & Type of $x_e$ & Type of $x_0$ \\
        \hline
        $x_e < x_0$ & Saddle & Unstable node \\
        \hline
        $x_e > x_0$ & Unstable node & Saddle \\
        \hline
        \end{tabular}
    \quad
        \begin{tabular}{|c|c|c|}
        \hline
        Condition & Type of $x_e$ & Type of $x_1$ \\
        \hline
        $x_e < x_1$ & Stable node & Saddle \\
        \hline
        $x_e > x_1$ & Saddle & Stable node \\
        \hline
        \end{tabular}
        \caption{Transcritical bifurcation at  $x_e = x_0$ (left) and $x_e = x_1$ (right).}
\label{tab:34}
\end{table}

\begin{figure}[h]
        \begin{overpic}[width=5cm]{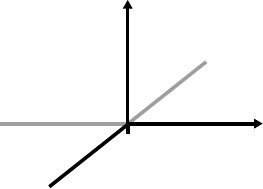}
            \put(41,67){y}
            \put(43,12){$(0,0)$}
            \put(90,15){$x_{e}-x_{0}$}
            \put(60,53){$y=\frac{-x_{0}+x_{1}}{x_{0}x_{1}}(x_{e}-x_0)$}
        \end{overpic}
\quad \quad \quad
        \begin{overpic}[width=5cm]{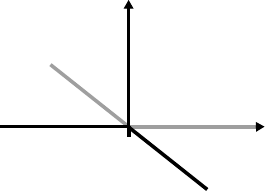}
            \put(41,67){y}
            \put(41,12){$(0,0)$}
            \put(90,15){$x_{e}-x_{1}$}
            \put(80,0){$y=\frac{x_{0}-x_{1}}{x_{0}x_{1}}(x_{e}-x_{1})$}
        \end{overpic}
    \caption{Transcritical bifurcation diagrams around $x_0=x_e$ (left) and $x_e=x_1$ (right). In black we represent the equilibria of saddle type, while in gray we represent the node type.} 
    \label{fig:transcritical}
\end{figure}
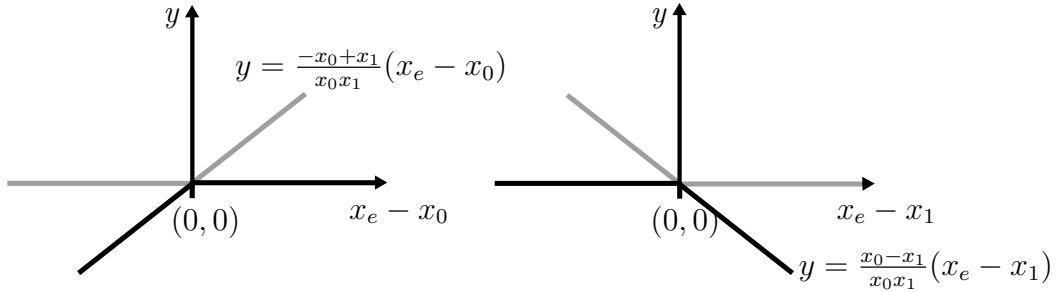
\end{pro}

\subsection{Slow-Fast phenomena}\label{subsec:slowfast}

For $0 < \varepsilon \ll 1$, systems of the form
\begin{equation}\label{eq:sp}
    \begin{aligned}
        \dot{x} &= f(x, y, \varepsilon),\\
        \dot{y} &= \varepsilon\, g(x, y, \varepsilon),
    \end{aligned}
\end{equation}
are said to exhibit a \textit{slow-fast structure}. The variable $x$ evolves on a fast time scale, with dynamics of order $\mathcal{O}(1)$, whereas $y$ evolves on a slow time scale, with dynamics of order $\mathcal{O}(\varepsilon)$. Setting $\varepsilon = 0$ freezes the slow dynamics ($\dot{y} = 0$), yielding the \textit{fast subsystem}. The equilibria of this subsystem form a curve 
\[
\mathcal{C} = \{(x, y) : f(x, y, 0) = 0\},
\]
called the \textit{critical manifold}. Linearization at a point $(x, y) \in \mathcal{C}$ typically produces one zero eigenvalue (in the slow direction) and one nonzero eigenvalue (in the fast direction). If the nonzero eigenvalue does not vanish along $\mathcal{C}$, the manifold is said to be \textit{normally hyperbolic}. In this case, Fenichel's Theorem guarantees the persistence of a locally invariant slow manifold $\mathcal{M}_\varepsilon$ located at a distance $\mathcal{O}(\varepsilon)$ from $\mathcal{C}$. Nearby trajectories are rapidly attracted to $\mathcal{M}_\varepsilon$ and subsequently evolve slowly along it.

Normal hyperbolicity can fail at isolated points of $\mathcal{C}$, for instance at turning points where the fast eigenvalue vanishes. Near such points, the slow manifold persists only within an exponentially thin region of width $\mathcal{O}(\ee^{-c/\varepsilon})$. This loss of normal hyperbolicity gives rise to the so-called \textit{canard phenomenon}, first identified by Beno\^it and colleagues in \cite{Ben1981,BenCal1981,Cal1981,DieDie1981} using nonstandard analysis, later analysed in a geometric framework by Dumortier and Roussarie~\cite{dumortier1996canard} for the van der Pol oscillator, and subsequently formalized by Krupa and Szmolyan~\cite{KrupaSzmolyan2001} for all generic turning points via the blow-up method. In this setting, canard solutions are trajectories that follow both the attracting and the repelling branches of the slow manifolds, connecting distinct dynamical regimes. These developments constitute a fundamental extension of Fenichel's geometric singular perturbation theory.

Having introduced the general slow-fast framework, we now study system~\eqref{eq:4} within this setting. 
The parameter $F$ acts as the small perturbation $\varepsilon$, and setting $F = 0$ freezes the slow dynamics ($\dot{y} = 0$), reducing the system to its fast subsystem. Every point on the curve
\[
\mathcal{C} = \bigl\{ (x, y) : f(x) = -\tfrac{(x - x_0)(x - x_1)}{x_0 x_1} \bigr\}
\]
is then a non-isolated equilibrium. Linearization at $(x_e, y_e) \in \mathcal{C}$ yields one nonzero eigenvalue in the $x$-direction and a zero eigenvalue in the $y$-direction, confirming that $x$ is the fast variable and that $\mathcal{C}$ is the critical manifold.

For small but positive $F$, the system becomes slow-fast, and Fenichel's Theorem ensures the existence of a nearby invariant slow manifold $\mathcal{M}_F$. Normal hyperbolicity fails precisely at the Hopf point
\[
x_e = \frac{x_0 + x_1}{2},
\]
where the fast eigenvalue vanishes. Near this point, $\mathcal{M}_F$ persists only within an exponentially thin region of width $\mathcal{O}(\ee^{-c/F})$.

System~\eqref{eq:4} thus exhibits both relaxation oscillations and a canard explosion. As the parameters vary, the limit cycle emerging from the Hopf bifurcation grows rapidly and eventually collides with the heteroclinic orbit. This abrupt transition can be understood by analysing the invariant manifolds.
The critical manifold $\mathcal{C}$ represents the limit set of these invariant manifolds as $F$ goes to $0$, and their linear parts are tangent to $\mathcal{C}$ at the saddle points. Consequently, if the invariant manifolds converge to the critical manifold in the singular limit $F = 0$, the slow manifolds necessarily coincide with these invariant manifolds for small $F$.
Moreover, the slow manifolds are not unique: they form a family lying within an exponentially small neighbourhood of $\mathcal{C}$. However, this non-uniqueness is negligible, since all such manifolds remain exponentially close to each other. As the limit cycle disappears into a heteroclinic connection without forming a secondary small loop, the trajectory corresponds to a \textit{headless canard}. Determining the parameter values for which such a canard occurs provides the conditions for the heteroclinic connection and elucidates the link between local canard dynamics and the global bifurcation.

The framework of singular perturbation theory is employed here because it allows us to obtain local information about the heteroclinic bifurcation near $(x_{e}=\frac{x_{0}+x_{1}}{2},\, F=0)$; 
in particular, it provides a way to compute the first-order approximation of the heteroclinic bifurcation curve
emerging from that point. To do so, our system is converted into a normal form ---through an appropriate translation and rescaling of variables---
in which the next result of Krupa and Szmolyan can be directly applied. 

\begin{tho}[\cite{KrupaSzmolyan2001}]\label{tho:KS}
Consider system \eqref{eq:sp} with a small parameter $\varepsilon>0$. Assume that $(x,y,\lambda,\varepsilon)=(0,0,0,0)$ is a generic fold point, namely:
\begin{equation*}
f = f_x = g = 0, \qquad 
f_{xx}\neq 0,\quad 
f_y\neq 0,\quad 
g_x\neq 0,\quad 
g_\lambda\neq 0
\quad\text{at }(0,0,0,0).
\end{equation*}

After suitable scalings, it can be written in the canonical form
\begin{equation}\label{eq:5}
\begin{aligned}
x' &= -y\,h_1 + x^2 h_2 + \varepsilon h_3,\\
y' &= \varepsilon\bigl(\pm x\,h_4 - \lambda h_5 + y h_6\bigr),
\end{aligned}
\end{equation}
where $h_3=\mathcal O(x,y,\lambda,\varepsilon)$ and  
$h_j = 1+\mathcal O(x,y,\lambda,\varepsilon),$ for $j=1,2,4,5,6$.

Suppose that the reduced problem admits a trajectory $x_0(t)$ connecting the attracting and repelling slow manifolds $S_a$ and $S_r$. Then there exist $\varepsilon_0>0$ and a smooth function $\lambda_c(\sqrt\varepsilon)$ on $[0,\varepsilon_0]$ with the following properties for every $\varepsilon\in(0,\varepsilon_0]$:
\begin{enumerate}[(i)]
\item The matching condition $\pi(q_{a,\varepsilon})=q_{r,\varepsilon}$ holds if and only if  
\[
\lambda=\lambda_c(\sqrt\varepsilon).
\]

\item The parameter $\lambda_c$ admits the expansion  
\[
\lambda_c(\sqrt\varepsilon)
= -\Bigl(\frac{a_1+a_5}{2}+\frac{A}{8}\Bigr)\varepsilon
+ \mathcal O(\varepsilon^{3/2}).
\]

\item The transition map $\pi$ is defined only for $\lambda$ in an interval of width $\mathcal O(\ee^{-c/\varepsilon})$ around $\lambda_c(\sqrt\varepsilon)$, for some $c>0$.

\item The corresponding canard is nondegenerate in the sense that
\begin{equation}\label{eq:monotonicity}
\left.\frac{\partial}{\partial\lambda}
\bigl(\pi(q_{a,\varepsilon})-q_{r,\varepsilon}\bigr)
\right|_{\lambda=\lambda_c(\sqrt\varepsilon)}
>0 .
\end{equation}
\end{enumerate}
\end{tho}

\begin{remark}
The monotonicity condition \eqref{eq:monotonicity} is a rotated vector field property: it ensures that the slow manifolds (equivalently, the invariant manifolds of the full system) move monotonically with respect to the parameter~$\lambda$.  Consequently, they can intersect at most once, yielding uniqueness of the canard solution.
\end{remark}

The Hopf point, located at the maximum of the $x$-nullcline parabola,
$$
\left(\frac{x_{0}+x_{1}}{2},\, \frac{(x_{0}-x_{1})^{2}}{4x_{0}x_{1}}\right),
$$
is translated to the origin. In the translated coordinates the system becomes
\begin{equation*}
    \begin{aligned}
        \dot{x}&=-\frac{x^{3}}{x_{1}x_{0}} 
        - \frac{(x_{0}+x_{1})x^{2}}{2x_{1}x_{0}} 
        - xy - \left(\frac{x_{0}+x_{1}}{2}\right)y,\\[4pt]
        \dot{y}&=F\!\left[\!\left(y + \frac{(x_{0}-x_{1})^{2}}{4x_{1}x_{0}}\right)x 
        + \left(\frac{x_{0}+x_{1}}{2} - x_{e}\right)y 
        + \frac{(x_{0}-x_{1})^{2}(x_{0}+x_{1}-2x_{e})}{8x_{1}x_{0}}\!\right]\!.
    \end{aligned}
\end{equation*}

\begin{pro}\label{pro:CCF}
The canonical form~\eqref{eq:5} is obtained by scaling the original variables and parameters as
\[
x = m\,\bar{x}, \qquad 
y = n\,\bar{y}, \qquad 
F = l\,\bar{F}, \qquad 
\tau = \bar{\tau}/t,
\]
where
\[
l = \frac{2(x_{0}+x_{1})}{x_{0}x_{1}(x_{1}-x_{0})^{2}}, 
\qquad m = 1, 
\qquad n = -\frac{1}{x_{1}x_{0}}, 
\qquad t = -\frac{x_{0}+x_{1}}{2x_{1}x_{0}}.
\]
The rescaled system becomes
\begin{equation}\label{eq:6}
    \begin{aligned}
        \dot{\bar{x}} &= \bar{x}^{2}\!\left(1+\frac{2\bar{x}}{x_{0}+x_{1}}\right)
        - \bar{y}\!\left(1+\frac{2\bar{x}}{x_{0}+x_{1}}\right),\\[4pt]
        \dot{\bar{y}} &= \bar{F}\!\left[\!\left(1 - \frac{4\bar{y}}{(x_{1}-x_{0})^{2}}\right)\bar{x}
        + \lambda\!\left(1 - \frac{4\bar{y}}{(x_{1}-x_{0})^{2}}\right)\!\right],
    \end{aligned}
\end{equation}
where $\lambda = x_{e} - \frac{x_{0}+x_{1}}{2}$ is the Hopf condition.
\end{pro}

Directly from system~\eqref{eq:6}, the functions $h_i$ for 
$i = 1,\ldots,6$ can be computed, which allows us to determine $\lambda_c(\sqrt{F})$ and
verify that all assumptions of Theorem~\ref{tho:KS} are satisfied.

\begin{coro}\label{cor:CC}
The canard curve of system~\eqref{eq:4} is
\[
\lambda_c\!\left(\sqrt{F}\right)
    = -\frac{1}{(x_{1}-x_{0})^{2}x_{0}x_{1}}\,F
      + \mathcal{O}(F^{3/2}),
    \qquad 
    \text{where }\lambda = x_{e}-\frac{x_{0}+x_{1}}{2}.
\]
\begin{proof}
From Proposition~\ref{pro:CCF}, the scaled system~\eqref{eq:6} has
\[
h_1 = h_2 = 1+\frac{2\bar{x}}{x_0+x_1}, \qquad
h_3 = h_6 = 0, \qquad
h_4 = h_5 = 1 - \frac{4\bar{y}}{(x_1-x_0)^2}.
\]

The coefficients in the canard curve are defined by
\[
a_1=\partial_x h_3(0),\quad 
a_2=\partial_x h_1(0),\quad
a_3=\partial_x h_2(0),\quad
a_4=\partial_x h_4(0),\quad
a_5=h_6(0).
\]
Evaluating at the origin gives
\[
a_1=a_4=a_5=0,
\qquad 
a_2=a_3=\frac{2}{x_0+x_1},
\]
and therefore
\[
A=-a_2+3a_3=\frac{4}{x_0+x_1}.
\]

Applying Theorem~\ref{tho:KS} yields
\[
\lambda_c(\sqrt{\bar F})
    = -\frac{1}{2(x_0+x_1)}\,\bar F
      +\mathcal O(\bar F^{3/2}).
\]

Using the scaling relation \(F=l\,\bar F\) we obtain
\[
\lambda_c(\sqrt{F})
    = -\frac{1}{(x_1-x_0)^2 x_0 x_1}\,F
      +\mathcal O(F^{3/2}).
\]
\end{proof}

\end{coro}

\begin{coro}
At the singular Hopf point $(\lambda,F) = (0,0)$, the canard curve and the
heteroclinic bifurcation curve have the same local expansion.
\end{coro}

\begin{proof}
As $F$ goes to $0$, both the invariant (stable and unstable) manifolds defining the heteroclinic connection and the  slow manifolds defining the canard solution converge to the same object--the critical manifold $\mathcal{C}$. 
Locally, this implies that they coincide up to analytic order: the invariant manifolds are unique, while the slow manifolds form a non-unique family lying within an exponentially thin neighbourhood of $\mathcal{C}$. 
Therefore, both constructions share the same local expansion near $(\lambda,F)=(0,0)$.
\end{proof}

The canard phenomenon captures a sharp yet continuous, parameter-dependent passage from small to large oscillations that emerges near critical thresholds of the system. The numerical study in \cite{JosVidBlaiErnest2021} examines both the transients generated by the heteroclinic connection and the slow-fast limit cycle, which appears as a self-sustained pattern of abrupt oscillations. Ecologically, these behaviours become especially pronounced for small values of the parameter $F$, which corresponds to a reduced consumer efficiency relative to vegetation turnover. In this regime, the vegetation is more sensitive and vulnerable, making the system prone to sudden large-amplitude oscillations when parameters approach the critical thresholds associated with the canard explosion and the heteroclinic bifurcation. Moreover, low values of $F$ induce a pronounced slow-fast structure: trajectories spend long times drifting along slow manifolds, which can be interpreted ecologically as metastable states where the system appears quasi-stationary on short time scales before undergoing rapid transitions. Complementary analyses of this slow-fast regime for the Bazykin--Berezovskaya model, where the slow time scale is associated with the resource dynamics rather than with the consumer, can be found in \cite{Borsotti_2026}.

\subsection{Heteroclinic bifurcation}\label{sec:heteroclinic}

A heteroclinic bifurcation is a global phenomenon in which the unstable
manifold of one equilibrium connects to the stable manifold of another,
forming a heteroclinic orbit for specific parameter values.  A small change in
the parameters breaks this connection, leading to a qualitative change in the
global dynamics--typically creating or destroying large-scale invariant
structures such as limit cycles.

In this section we characterize the heteroclinic connection and use it to describe qualitatively the heteroclinic bifurcation curve in parameter space. In our system this curve identifies the parameter values at which the limit cycle collides with the polycycle and disappears, leading to ecological collapse. The bifurcation is of codimension one: although, in principle, up to two heteroclinic connections may occur, one of them always lies on the invariant axis $y=0$. Consequently, only a single nontrivial heteroclinic connection plays a dynamical role.

The first step is to verify that such a heteroclinic connection does indeed occur. Corollary~\ref{cor:unstableconnection} ensures that the connection is confined to a bounded interval of parameter values, implying that for some parameter choice it must be realized. Geometrically, this creates a one limit cycle region in parameter space, bounded on one side by the Hopf bifurcation curve (see Section~\ref{sec:Hopf}) and on the other by the heteroclinic bifurcation curve described here. 

To analyse the heteroclinic bifurcation curve, we examine the return map around the polycycle. Its leading-order behaviour depends only on the hyperbolicity ratio, which determines whether the connection is attracting or repelling. The study of the return map is typically used to bound cyclicity; however, in our setting, Corollary~\ref{coro:uniqueness} already ensures that no limit cycle can persist at the heteroclinic parameter value.

\subsubsection{Return map near a polycycle and the role of the hyperbolicity ratio}

Let $\rho>0$ denote a transversal coordinate measuring the distance to the
heteroclinic connection. For each saddle $S_i$ ($i=0,1$), in suitable local
coordinates the transition map between incoming and outgoing sections has the
form
\[
\Pi_i(\rho)=C_i\,\rho^{\sigma_i}\,(1+o(1)),
\qquad
\sigma_i=-\frac{\lambda_-^i}{\lambda_+^i}>0,
\]
where $\lambda_+^i>0$ and $\lambda_-^i<0$ denote the unstable and stable
eigenvalues of $S_i$. Composing the two local maps yields the global return map
\[
\Pi(\rho)=C\,\rho^{f_{\lambda}}\,(1+o(1)),
\qquad
C=C_1 C_0^{\sigma_1}>0,
\qquad
f_{\lambda}=\sigma_0\sigma_1.
\]

Thus,
\[
f_{\lambda}
=
\Bigl(-\frac{\lambda_-^0}{\lambda_+^0}\Bigr)
\Bigl(-\frac{\lambda_-^1}{\lambda_+^1}\Bigr)
\]
is the hyperbolicity ratio. For sufficiently small
$\rho>0$, the value of $f_{\lambda}$ determines whether the polycycle is locally
attracting or repelling:
\[
\text{ If } 0<f_{\lambda}<1 \text{ then, the connection is attracting and}
\]
\[
\text{ if } f_{\lambda}>1 \text{ then, the connection is repelling}.
\]
The case $f_{\lambda}=1$ corresponds to a nonhyperbolic connection.

\subsubsection*{Application to system \eqref{eq:4}}
For our system, the hyperbolicity ratio becomes
$$
f_{\lambda}
=
\frac{F(x_{0}-x_{e})}{\frac{x_{1}-x_{0}}{x_{1}}}
\cdot
\frac{\frac{x_{0}-x_{1}}{x_{0}}}{F(x_{1}-x_{e})}
=
\frac{(x_{e}-x_{0})x_{1}}{(x_{1}-x_{e})x_{0}}
> 0.
$$
As a function of the unfolding parameter $x_{e}$, this expression is strictly
increasing on the interval $x_{0}<x_{e}<x_{1}$. Hence, the nonhyperbolic condition
$f_{\lambda}=1$ occurs precisely when
$$
x_{c}:=x_e=\frac{2x_{0}x_{1}}{x_{0}+x_{1}}>0.
$$

Moreover, this critical value satisfies
$$
x_{0}
<
x_c
<
\frac{x_{0}+x_{1}}{2}.
$$

Along the heteroclinic bifurcation curve $x_{e}=x_{e}^{h}(F)$, the unique stable limit cycle (see Corollary~\ref{coro:uniqueness}) of the system has already collided with the polycycle and no periodic orbit remains. Since the coexistence equilibrium is repelling for all
$x_{0}<x_{e}<\tfrac{x_{0}+x_{1}}{2}$ (see Table~\ref{tab:stability_cases}), the heteroclinic connection must be attracting. This agrees with the condition $f_{\lambda}\ge 1$, which characterizes the parameter values where the polycycle is locally attracting for small transversal distance $\rho>0$.

\begin{coro}\label{cor:unstableconnection}
The heteroclinic connection must occur in the interval
$$
\frac{2x_0 x_1}{x_0 + x_1} \leq x_e < \frac{x_0 + x_1}{2}.
$$
\begin{proof}
From the previous discussion, the heteroclinic polycycle must be attracting, since it
coexists with a repelling coexistence equilibrium. This corresponds to the condition
$f_{\lambda}\ge 1$, which is equivalent to
$$
x_e \;\ge\; \frac{2x_0 x_1}{x_0 + x_1}.
$$
Moreover, the heteroclinic bifurcation can only occur once the limit cycle has emerged, so
necessarily
$$
x_e < \frac{x_0 + x_1}{2}.
$$
Combining both bounds yields the claimed interval.
\end{proof}
\end{coro}

\begin{pro}\label{pro:dieLC}
	The limit cycle dies at the heteroclinic connection.
	\begin{proof}
	Since there is only one Andronov--Hopf bifurcation and we now that only one limit cycle can appear, it can only disappear either through another Hopf (i.e.\ a change of stability of the coexistence
equilibrium) or at the heteroclinic connection. In the interval
$$\frac{2x_0x_1}{x_0+x_1} \le x_e < \frac{x_0+x_1}{2},$$
no change of stability occurs, and therefore the limit cycle can only vanish
by merging with the heteroclinic connection.
	\end{proof}
\end{pro}

 Since the heteroclinic curve $x_{e}(F)$ cannot be written explicitly, it is
useful to determine sharp bounds for its possible values. Corollary~\ref{cor:unstableconnection} shows that
its lower bound is dictated by the condition $f_{\lambda}=1$, and the following lemma, proved in Appendix~\ref{sec:optimalbound}, shows the optimality:

\begin{lema}\label{lem:optimal_bound}
The heteroclinic bifurcation curve $x_e^{\mathrm{h}}(F)$ satisfies
$$
\lim_{F\to\infty} x_e^{\mathrm{h}}(F) = x_e^\ast := \frac{2x_0x_1}{x_0+x_1}.
$$
Hence, $x_e^\ast$ is the optimal lower bound of the heteroclinic bifurcation curve.
\end{lema}

\begin{coro}\label{cor:NLC}
	The region without limit cycles is
	$$
	\Omega_{NLC}=\left\{(x_e, f) \in (x_0, x_1) \times \mathbb{R}^{+} \,\middle|\, \left( x_0 \leq x_e \leq \frac{2x_0 x_1}{x_0 + x_1}\right)\cup \left(\frac{x_0 + x_1}{2} \leq x_e \leq x_1\right) \right\}.
	$$
	\begin{proof}
		For $x_e < \frac{2x_0 x_1}{x_0 + x_1}$, it is known that the heteroclinic connection cannot occur. Therefore, no mechanism for the creation or destruction of limit cycles exist in this region. From Proposition~\ref{th:BD}, it is known that there is no limit cycle for $\frac{x_0 +x_1}{2}\leq x_e < x_1$.
	\end{proof}
\end{coro}

\begin{lema}\label{lemma:analytic}
	The curve $x_{e}^{h}(F)$ is analytic except for $F=0$.
\begin{proof}
The proof follows the classical analytic approach used by Perko for the
Bogdanov--Takens heteroclinic curve \cite{Perko1992}, adapted to the parameters
$\vec{\mu}=(x_e,F)$ of our system.

Since the vector field is analytic in $(x,y,\vec{\mu})$, the analytic
dependence theorem implies that solutions, and therefore the global stable and
unstable manifolds of the saddle points $x_0$ and $x_1$, depend analytically on
$\vec{\mu}$.  For each $\vec{\mu}$ we consider their first intersection with the
vertical line $x=x_e$ and define an analytic Melnikov-type function
$D(\vec{\mu})$ giving the signed distance between these two manifolds on this
section.  A heteroclinic connection occurs precisely when $D(\vec{\mu}_h)=0$ for
some heteroclinic parameter $\vec{\mu}_h=(x_e^h,F^h)$.

The rotated vector field property (see more details in Appendix~\ref{sec:rvf}) ensures that the
manifolds move monotonically with respect to $x_e$, so
$\partial_{x_e}D(\vec{\mu}_h)\neq 0$.  Hence, the implicit function
theorem applies, yielding a local analytic function $x_e^{h}(F)$  such that $D(x_e^{h}(F),F)=0$ for all $F$ sufficiently close to any heteroclinic value $F^h\neq 0$.  Since the
heteroclinic connection is unique and persists smoothly for $F>0$, these local
analytic branches glue together, showing that $x_e^{h}(F)$ is analytic for every
$F>0$.

Finally, Corollary~\ref{cor:CC} shows that the canard curve
$\lambda_c(\sqrt{F})$ is non-analytic at $F=0$, and since this curve provides a
local representation of the heteroclinic bifurcation curve near $(\lambda,F)=(0,0)$,
the function $x_e^{h}(F)$ cannot be analytic at $F=0$.  
\end{proof}
\end{lema}

\begin{pro}
The heteroclinic connection bifurcation curve can be expressed as $F^{h}(x_{e})$ and it is monotone.
\end{pro}

\begin{proof}
From boundedness and the fact that the invariant manifolds $W^{s}(x_{0})$ and $W^{u}(x_{1})$
are bounded from below by the prey nullcline
\[
y_{0}(x)=-\frac{(x-x_{0})(x-x_{1})}{x_{0}x_{1}},
\]
it follows that their first intersections with the vertical line $x=x_{e}$ are finite.
Moreover, in the region $\Omega$ (introduced in Section~\ref{subsec:rvfmodel}) between the equilibria and
these intersections one has $y>y_{0}(x)$ and therefore $\dot x=x(y_{0}(x)-y)<0$, so $x(t)$ is strictly
monotone. Hence, each of the relevant separatrix branches can be written, up to its first
intersection with $x=x_{e}$, as a graph $y=y(x)$.

Along such graphs we have
\[
\frac{dy}{dx}=\frac{\dot y}{\dot x}
=\frac{F\,y(x-x_{e})}{x\bigl(y_{0}(x)-y\bigr)}.
\]

Fix $x_e$ and consider $0<F_{1}<F_{2}$. Let $y_{1}(x)$ and $y_{2}(x)$ be the graphs of the same
separatrix branch (either of $W^{s}(x_{0})$ or of $W^{u}(x_{1})$) corresponding to $F_1$ and $F_2$.
Near the saddle one has $y_{2}(x)>y_{1}(x)$. Suppose by contradiction that the graphs intersect, and
let $x^{*}$ be the first contact point (moving from the saddle towards $x=x_e$) such that
$y_{1}(x^{*})=y_{2}(x^{*})$. Then necessarily $y_{2}'(x^{*})\le y_{1}'(x^{*})$. However, evaluating the
above equation at $x^{*}$ yields
\[
y_{2}'(x^{*})-y_{1}'(x^{*})
=(F_{2}-F_{1})\frac{y(x^{*}-x_{e})}{x^{*}\bigl(y_{0}(x^{*})-y\bigr)},
\]
whose sign is fixed in $\Omega$ and therefore cannot vanish. This contradicts the inequality at the
first contact point. Hence the two graphs cannot intersect and are strictly ordered. In particular,
the height of the first intersection with $x=x_e$ depends strictly monotonically on $F$.

Denote by $S(F,x_e)$ and $U(F,x_e)$ the heights of the first intersections of the relevant branches
of $W^{s}(x_{0})$ and $W^{u}(x_{1})$ with the section $x=x_e$. A heteroclinic connection occurs if and
only if $S(F,x_e)=U(F,x_e)$. Since, for fixed $x_e$, the branches of $W^{s}(x_{0})$ and $W^{u}(x_{1})$
move monotonically with respect to $F$ (in opposite directions), the functions
$F\mapsto S(F,x_e)$ and $F\mapsto U(F,x_e)$ are strictly monotone. Consequently, the equality
$S(F,x_e)=U(F,x_e)$ can be satisfied for at most one value of $F$, and hence the intersection between
$W^{s}(x_{0})$ and $W^{u}(x_{1})$ is unique when $x_e$ is fixed. Therefore, the heteroclinic
bifurcation set is the graph of a function $F=F^{h}(x_e)$. Due to the obtained result and Proposition~\ref{pro:xe_f}, the function is one-to-one. Since the heteroclinic bifurcation curve can be written as $x_e^{h}(F)$ and this function is analytic (by Lemma~\ref{lemma:analytic}) and one-to-one, its inverse $F^{h}(x_e)$ is also analytic by the
analytic inverse function theorem. In particular, $F^{h}(x_e)$ is continuous and bijective on an
interval, and therefore it is strictly monotone.

\end{proof}

\begin{tho}\label{th:HBC}
The heteroclinic bifurcation set consists of a single curve that satisfies the following properties:
\begin{enumerate}[(i)]
    \item \textit{Graph property.}  
    The heteroclinic bifurcation curve is the graph of a strictly monotone function $F^{h}(x_e)$ (equivalently $x_{e}^{h}(F)$).

    \item \textit{Analyticity.}  
    The function \(x_e^{h}(F)\) is analytic for all \(F>0\) and loses analyticity only at
    \(F=0\).

    \item \textit{Bounds.}  
    For every \(F>0\), the function \(x_e^{h}(F)\) is optimally bounded by
    $$\frac{2x_0 x_1}{x_0 + x_1}\;\le\;x_e^{h}(F)\;<\;\frac{x_0 + x_1}{2}.$$

    \item \textit{Behaviour near $F=0$.}  
    The curve begins as
    $$ \lambda_c\!(\sqrt{F})
        = 
        -\frac{1}{(x_1 - x_0)^2 x_0 x_1}\,F
        + \mathcal{O}\!\left(F^{3/2}\right).$$   
    \item \textit{Asymptotic behaviour as $F\to\infty$.}  
    The heteroclinic curve satisfies
        $$\lim\limits_{F\rightarrow \infty}x_e^{h}(F)=\dfrac{2x_0x_1}{x_{0}+x_{1}}.$$
\end{enumerate}
\end{tho}

\subsubsection{Effect of habitat destruction}\label{subsec:habitatdestruction}

The parameter space in $(x_e,F)$ naturally splits into three generic ecological
regimes (see Figure~\ref{fig:collapse}):  
(i) a \emph{collapse region}, where both the resource and the consumer extinguish;  
(ii) a \emph{potential persistent oscillation regime}, where self-sustained oscillations may occur
for certain values of $(x_{e},F)$; and  
(iii) a \emph{static coexistence region}, where the resource and consumer tend to the stable coexistence equilibrium
$(x_e,y_e)$.

\medskip
\noindent
We first analyse the collapse region.  
As shown above, the system collapses for all values of~$F$ whenever
$$
x_0 < x_e < x_c,
\qquad
x_c=\frac{2x_0x_1}{x_0+x_1}.
$$
Thus, the interval $(x_0,x_c)$ corresponds to the left grey region in
Figure~\ref{fig:collapse}.  
Its relative size within $(x_0,x_1)$ is
$$
R_c=\frac{x_c-x_0}{x_1-x_0}
    =\frac{x_0}{x_0+x_1}
    =\frac12\!\left(
      1-\sqrt{1-\frac{4\varepsilon}{\alpha(1-D)^2}}
      \right),
$$
after expressing $x_0$ and $x_1$ in the original parameters.
The admissible range for habitat destruction is
$$
0 < D < 1 - 2\sqrt{\varepsilon/\alpha},
$$
which ensures $x_0<x_1$ (see Section~\ref{subsubsec:saddlenode}).  
Within this range, $R_c$ increases monotonically with~$D$ and satisfies
that $R_c$ goes to $\tfrac12$ as $D$ approaches the saddle-node value
$D_{\mathrm{SN}}=1-2\sqrt{\varepsilon/\alpha}$.  
Thus, habitat destruction simultaneously contracts the interval $(x_0,x_1)$ and
enlarges the portion of $x_e$ leading unavoidably to collapse.

\medskip
\noindent
Next, consider the oscillatory regime.  
Persistent oscillations may only occur for
$$
x_c < x_e < x_H,
\qquad x_H=\frac{x_0+x_1}{2},
$$
whose relative size is
$$
R_o=\frac{x_H-x_c}{x_1-x_0}.
$$
As $D$ increases, $x_c$ moves rightwards and $(x_0,x_1)$ contracts, causing
$R_o$ to decrease monotonically until it vanishes at the saddle-node threshold.  Points in this region lead to exactly two possible outcomes: either one
persistent oscillation or extinction, depending on whether the parameters lie
between the Hopf and the heteroclinic curves or not.

\medskip
\noindent
Finally, the static coexistence region corresponds to
$$
R_s=\frac{x_1-x_H}{x_1-x_0}=\frac12.
$$
This fraction is independent of habitat destruction~$D$ because $x_H$ always lies
at the midpoint of the interval $(x_0,x_1)$.  
Hence, while $R_c$ and $R_o$ vary with $D$, the static coexistence region remains
a constant half of the parameter space.

\medskip
\noindent
In summary, the three relative sizes satisfy
$$
R_c + R_o + R_s = 1.
$$
Increasing habitat destruction compresses the oscillatory region and expands the
collapse region, while the static coexistence region remains fixed.  
As $D$ approaches the saddle-node threshold, the oscillatory window disappears,
signalling a sharp loss of ecological resilience.

 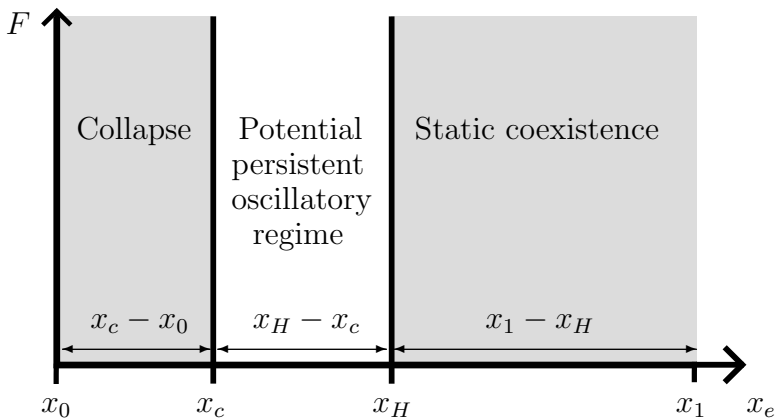
\begin{figure}[h]
	\centering
	\begin{overpic}[width=0.58\textwidth]{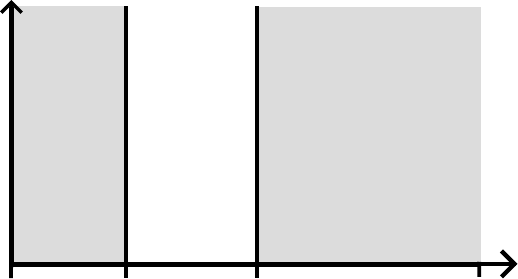}
		\put(-0,-4){$x_{0}$}
		\put(-5,50){$F$}
        \put(5,35){\text{Collapse}}
        \put(28,35){\text{Potential}}
        \put(27,30){\text{persistent}}
        \put(27,25){\text{oscillatory }}
        \put(30,20){\text{regime}}
		\put(47,-4){$x_{H}$}
        \put(53,35){\text{Static coexistence}}
		\put(22,-4){$x_{c}$}
		\put(90,-4){$x_{1}$}
        \put(3,5){\vector(1,0){21}}
\put(24,5){\vector(-1,0){21}}
\put(7,8){$x_{c}-x_{0}$}
		\put(100,-4){$x_{e}$}
        \put(25,5){\vector(1,0){24}}
        \put(49,5){\vector(-1,0){24}}
        \put(30,8){$x_{H}-x_{c}$}
        \put(51,5){\vector(1,0){42}}
        \put(92,5){\vector(-1,0){42}}
        \put(63,8){$x_{1}-x_{H}$}
	\end{overpic}
	\vspace{1cm}
    	 \captionsetup{width=\textwidth}
	   \caption{
Parameter space in $(x_{e},F)$ showing the different ecological regimes.
The grey region $\Omega_{\mathrm{NLC}}$ corresponds to the \emph{no-limit-cycle
regime} (see Corollary~\ref{cor:NLC}); it includes both the collapse region, where
the resource and the consumer unavoidably go extinct, and the static coexistence
regime, where $(x_e,y_e)$ is a stable equilibrium. The central white region represents the \emph{potential persistent oscillation
regime}. Points in this region lead to exactly two possible outcomes: either one
persistent oscillation or extinction, depending on whether the parameters lie
between the Hopf and the heteroclinic curves or not. As habitat destruction $D$ increases, the interval $(x_0,x_1)$ contracts, the
collapse region within $\Omega_{\mathrm{NLC}}$ expands, and the potential oscillatory
region shrinks, eventually disappearing when the saddle-node bifurcation is
reached.
}
\label{fig:collapse}
\end{figure}

If one focuses in the potential persistent oscillatory regime 

\subsection{Description of parameter space}
In this subsection we describe the parameter space $(x_e,F)$, illustrating how the different dynamical regimes are separated by the bifurcation curves previously characterised. For each region we summarise the corresponding phase-space behaviour together with its ecological interpretation.

\begin{enumerate}[(a)]

\item 
$\Omega_{1}
=\Bigl\{(x_e,F)\in(x_0,x_1)\times\mathbb{R}^{+}\,\bigm|\,
\frac{x_0+x_1}{2}\le x_e < x_1\Bigr\}.$ 
All trajectories in the basin of attraction tend
to this equilibrium, leading to a stable coexistence of the consumer and the
vegetation.

\item 
$\Omega_{2}
=\Bigl\{(x_e,F)\in(x_0,x_1)\times\mathbb{R}^{+}\,\bigm|\,
\frac{2x_0x_1}{x_0+x_1}\le x_e < \frac{x_0+x_1}{2}
\ \text{ and }\ x_e > x_e(F)\Bigr\}.$  
All trajectories in the basin of
attraction are drawn to a unique stable limit cycle. This region corresponds to sustained predator-prey oscillations.

\item 
$\Omega_{3}
=\Bigl\{(x_e,F)\in(x_0,x_1)\times\mathbb{R}^{+}\,\bigm|\,
\frac{2x_0x_1}{x_0+x_1}\le x_e < \frac{x_0+x_1}{2}
\ \text{ and }\ x_e = x_e(F)\Bigr\}.$  
 All trajectories in the basin of attraction tend to the heteroclinic connection. This phase portrait is structurally unstable, as any perturbation of $(x_e,F)$
destroys the connection, tending to either a limit cycle o collapse.

\item 
$\Omega_{4}
=\Bigl\{(x_e,F)\in(x_0,x_1)\times\mathbb{R}^{+}\,\bigm|\,
\frac{2x_0x_1}{x_0+x_1} \le x_e < \frac{x_0+x_1}{2}
\ \text{ and }\ x_e < x_e(F)
\Bigr\}.$  All trajectories tend to the extinction point $(0,0)$. This corresponds to ecological collapse: oscillations grow in period and amplitude until the consumer and vegetation both go extinct.
\end{enumerate}

In the regions $\Omega_{1}$, $\Omega_{2},$ and $\Omega_{3}$, the basin of attraction
is bounded by the invariant axis $y=0$ and the stable manifold of $x_{0}$, which
forms the upper invariant barrier separating collapse from persistence. Any
trajectory that starts within this basin necessarily has its forward-time limit
set inside it: in $\Omega_{1}$ it converges to the stable coexistence equilibrium,
in $\Omega_{2}$ to the stable limit cycle, and in $\Omega_{3}$ it approaches the
heteroclinic connection.

\begin{figure}[h]
	\centering
	\begin{overpic}[width=0.5\textwidth]{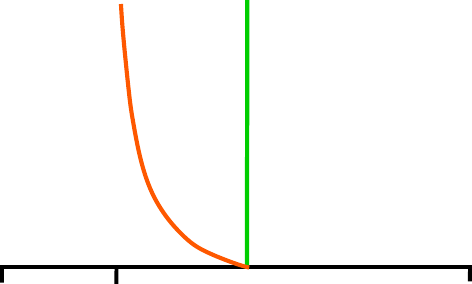}
		\put(-0.5,-4){$x_{0}$}
		\put(20,-4){$\frac{2x_0 x_1}{x_{0}+x_{1}}$}
		\put(48,-4){$\frac{x_{0}+x_{1}}{2}$}
		\put(98,-4){$x_{1}$}
        \put(78,44){$\Omega_{1}$}
        \put(38,44){$\Omega_{2}$}
        \put(12,24){$\Omega_{4}$}
	\end{overpic}
	\vspace{0.3cm}
    	 \captionsetup{width=\textwidth}
	\caption{Parameter space $(x_{e},F)$. The orange curve of the form $x_{e}(F)$ corresponds to the heteroclinic bifurcation curve. The green curve corresponds to the Hopf bifurcation.}
\end{figure}
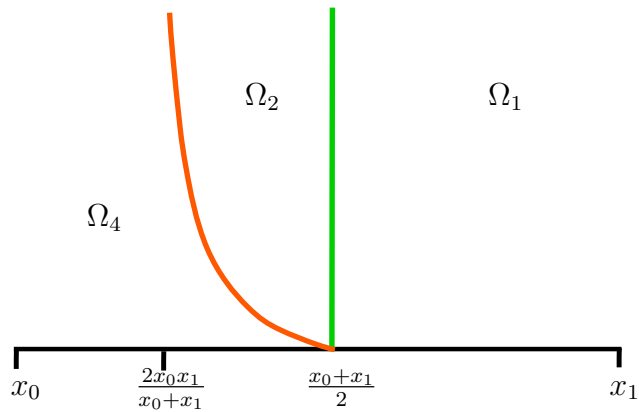

\vspace{0.5cm}

\begin{figure}[h]
    \centering

    \begin{minipage}{0.95\textwidth}
        \centering
        \begin{subfigure}[t]{0.2\textwidth}
            \centering
            \begin{overpic}[width=\textwidth]{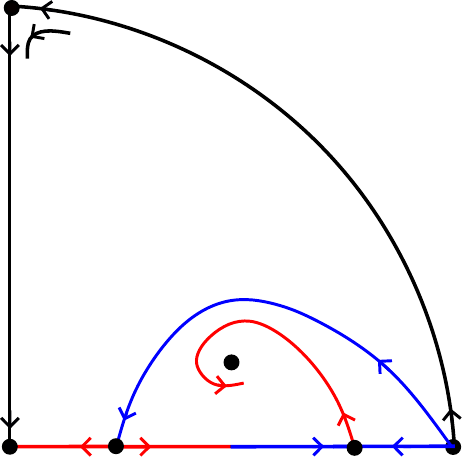}
            \end{overpic}
            \caption{$\Omega_{1}$.}
            \label{fig:om1}
        \end{subfigure}
        \hspace{0.05\textwidth} 
        \begin{subfigure}[t]{0.2\textwidth}
            \centering
            \begin{overpic}[width=\textwidth]{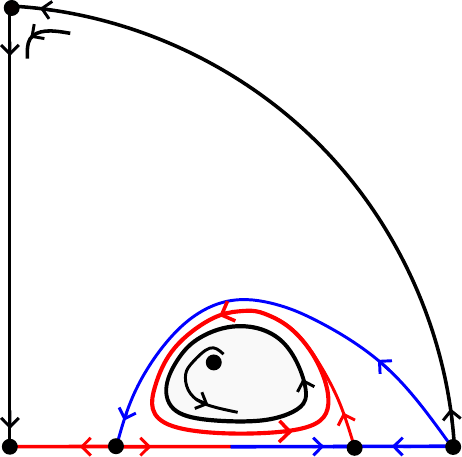}
            \end{overpic}
            \caption{$\Omega_{2}$.}
            \label{fig:om2}
        \end{subfigure}
        \hspace{0.05\textwidth}
        \begin{subfigure}[t]{0.2\textwidth}
            \begin{overpic}[width=\textwidth]{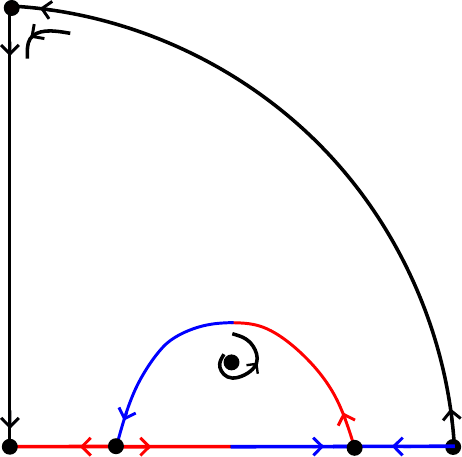}
            \end{overpic}
            \caption{$\Omega_{3}$.}
            \label{fig:om3}
        \end{subfigure}
        \hspace{0.05\textwidth}
        \begin{subfigure}[t]{0.2\textwidth}
            \centering
            \begin{overpic}[width=\textwidth]{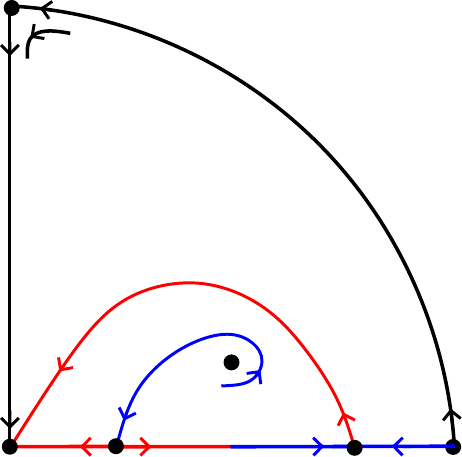}
            \end{overpic}
            \caption{$\Omega_{4}$.}
            \label{fig:om4}
        \end{subfigure}
    \end{minipage}

    \captionsetup{width=\textwidth}
    \caption{Phase portraits of all the regions in the parameter space $(x_{e},F)$, where $F > 0$ and $x_0 < x_e < x_1$. Stable and unstable manifolds are colored in blue and red, respectively.}
    \label{fig:pporiginal}
\end{figure}

\subsection{Piecewise linear model}\label{PWLM}
In this section we introduce a piecewise linear (PWL) approximation of the 
original model.  
The motivation for this approach is twofold.  
First, the PWL system allows us to write the solutions of both vector fields 
$Z_{1}$ and $Z_{2}$ explicitly, which provides complete analytical control of the 
trajectories.  
Second, global phenomena such as the heteroclinic bifurcation become fully 
tractable in this setting, enabling us to derive the bifurcation curve in closed 
form and to understand its geometry in detail.  
Although the PWL model is a simplification, we show below that it reproduces the 
same qualitative behaviour as the nonlinear system.  
Moreover, PW systems are known to exhibit a rich variety of generic bifurcations (see, for instance, \cite{KuzRinGra2003}), and the assumption of linearity makes them 
an ideal framework for uncovering the mechanisms that generate the transitions 
observed in the original model.

\medskip

We work under genericity conditions for which a coexistence equilibrium is present, 
and we focus on the region $x_{0}<x_{e}<x_{1}$.  
To analyse the dynamics we construct a PWL system that preserves the saddle points 
$x_{0}$ and $x_{1}$ together with their eigenvalues and eigenvectors.  
We use the vertical line 
\[
\Sigma=\{(x,y)\in\mathbb{R}^{+}\times\mathbb{R}^{+} \mid x = x_e\}
\]
as separation manifold between the two linear regions.  
The resulting PWL vector field is the following:

\begin{equation*}    \label{eq:pwlinear}
    Z=\begin{cases}
        Z_{1}\text{ if } (x,y)\in\Sigma_{1}=\{x_{0}<x<x_{e}\},\\
        Z_{2}\text{ if } (x,y)\in\Sigma_{2}=\{x_{e}<x<x_{1}\},
    \end{cases}
\end{equation*}
where
\begin{equation*}
    Z_{1}=\begin{cases}
        \dot{x}=\frac{(x_{1}-x_0)(x-x_{0})}{x_1}-x_0 y,\\
        \dot{y}=F(x_{0}-x_{e})y,
    \end{cases}
\end{equation*}
and
\begin{equation}
    Z_{2}=\begin{cases}
        \dot{x}=\frac{(-x_{1}+x_0)(x-x_{1})}{x_0}-x_1 y,\\
        \dot{y}=F(x_{1}-x_{e})y.
    \end{cases}
\end{equation}

\subsubsection{Qualitative features of the PWL system}

This section analyses the changes in the qualitative behaviour of the PWL system. We focus on understanding three main aspects: the pseudo-Hopf bifurcation, the heteroclinic bifurcation, and the no-return region.

\vspace{0.3cm}
\textbf{Pseudo-Hopf Bifurcation}. This bifurcation is analogous to the Hopf bifurcation, with the difference that the limit cycle emerging from the change of stability of an equilibria, is instead a pseudoequilibria, a zero of the separation manifold. When studying the PWL vector field, one observes that the $x$-nullclines of $Z_1$ and $Z_2$ provide essential information about the monodromy. In particular, a double tangency occurs when these nullclines intersect in $\Sigma$, providing a well-known mechanism for the emergence of limit cycles. This bifurcation is denoted as generic pseudo-Hopf  (also denoted as $II_2$ in \cite{KuzRinGra2003})  and resembles the classical Andronov--Hopf bifurcation; however, instead of exchanging the stability of an equilibrium, it changes the stability of the sliding segment. The crossing limit cycle inherits stability opposite to that of the sliding segment. This phenomenon was originally described in \cite{Fil1988} and studied further in \cite{KuzRinGra2003}.

To analyse the dynamics near the discontinuity, we define the sliding vector field on $\Sigma$. After translating the double tangency point
\[
\left(\lambda + \frac{x_0 + x_1}{2}, \frac{(x_1 - x_0)^2}{2 x_0 x_1} \right)
\]
to the origin, the sliding vector field becomes:
\begin{equation*}
Z^s =
\begin{pmatrix}
0 \\
-\dfrac{
\left( (x_1 y + \lambda)x_0^2 + (x_1^2 y + 2\lambda(y - 1)x_1 + 2\lambda^2)x_0 + x_1 \lambda(2\lambda + x_1) \right) C F }{ 2 x_1 x_0 \left( (x_1 y + \lambda)x_0 + x_1 \lambda \right)}
\end{pmatrix},
\end{equation*}
where $C = \dfrac{x_0^2 + x_1^2}{2} + x_1(y - 1)x_0$ and $\lambda = x_e - \frac{x_0 + x_1}{2}$ measures the distance from the Hopf point.

The tangency points between the flow and the discontinuity are:
\[
T_1 = \frac{(x_1 - x_0)}{x_1 x_0} \lambda, \qquad T_2 = -\frac{(x_1 - x_0)}{x_1 x_0} \lambda.
\]
When $\lambda = 0$, both tangencies coincide and there is no sliding region.

The pseudoequilibrium is located at
\[
P_{\lambda} = -\lambda \frac{2\lambda(x_0 + x_1) + (x_1 - x_0)^2}{x_0 x_1 (2\lambda + x_0 + x_1)}.
\]
For small $\lambda$, we find $P_{\lambda} < 0$ if $\lambda > 0$ and $P_{\lambda} > 0$ if $\lambda < 0$. Evaluating the vector field at this point,
\[
Z^s|_{y = P_{\lambda}} = \frac{(x_0 + x_1)(2\lambda - x_1 + x_0)(2\lambda + x_1 - x_0(2\lambda + x_1 + x_0))F}{16\lambda x_0 x_1},
\]
shows that $P_{\lambda}$ is unstable for $\lambda < 0$ and stable for $\lambda > 0$.

The PWL system admits first integrals, which facilitates the study of its return map:
\begin{equation*}
H := \begin{cases}
H_1 = \frac{\left(-\frac{F x_1^3}{2} - F(x - x_0 + 2\lambda) x_1^2 + \left(\left(-\frac{x_0}{2} + x + \lambda\right)(x_0 - 2\lambda)F + 2x_0 y - 2x - 2\lambda\right)x_1 + 2x_0(x + \lambda)\right) A}{-x_1^2 F + (-2 + (x_0 - 2\lambda)F)x_1 + 2x_0},
\\
H_2 = \frac{\left(\frac{x_0^3 F}{2} + F(x - x_1 + 2\lambda)x_0^2 + \left(-\left(-\frac{x_1}{2} + x + \lambda\right)(x_1 - 2\lambda)F - 2x_1 y + 2x + 2\lambda\right)x_0 - 2x_1(x + \lambda)\right) B}{x_0^2 F + (2 + (2\lambda - x_1)F)x_0 - 2x_1}
\end{cases}
\end{equation*}
with
\begin{align*}
A &= \left(x_0^2 + 2x_1(y - 1)x_0 + x_1^2\right)^{\frac{2(x_0 - x_1)}{(x_0 - 2\lambda - x_1)F x_1}}, \\
B &= \left(x_0^2 + 2x_1(y - 1)x_0 + x_1^2\right)^{\frac{2(x_0 - x_1)}{(-x_1 + 2\lambda + x_0)F x_0}}.
\end{align*}

\begin{lema}
There is a pseudo-Hopf bifurcation at $x_e = \frac{x_0 + x_1}{2}$. The pseudoequilibrium becomes a weak focus of order 1, and for $\lambda < 0$, a stable limit cycle emerges. The first Lyapunov constant is
\[
V_1 = -\frac{8}{3F(x_1 - x_0)}.
\]
\end{lema}

\begin{remark}
A unique limit cycle is unfolded from the pseudo-Hopf. Moreover, the second Lyapunov constant vanishes. More concretely, it satisfies $V_2 = V_1^{2}$, i.e. it vanishes when $V_1=0$.
\end{remark}

\textbf{Heteroclinic Bifurcation Curve.} The heteroclinic bifurcation curve represents the same global phenomenon as in the 
classical smooth ordinary differential equations setting, with the key difference that the intersection of the 
invariant manifolds now occurs at the separation manifold $\Sigma$.  
The bifurcation takes place when the stable manifold of $x_0$ and the unstable 
manifold of $x_1$ meet on $\Sigma$, producing a global connection. As in the original model, the heteroclinic bifurcation in the PWL system is a 
codimension-one phenomenon.  
This is because there is only one nontrivial global connection between the saddles 
$x_{0}$ and $x_{1}$: once their invariant manifolds meet on $\Sigma$, the 
heteroclinic loop is completely determined.  
Hence, a single parameter is sufficient to unfold the bifurcation. 
In the PWL system this connection can be fully analysed, since both vector fields 
$Z_{1}$ and $Z_{2}$ admit explicit first integrals.  
As in the original nonlinear model, the heteroclinic connection triggers the 
destruction of the limit cycle, whose period becomes unbounded as it approaches the connection. Moreover, systems~\eqref{eq:4} and \eqref{eq:pwlinear} have the same value of the hyperbolicity ratio. Hence, the connection is stable, see Figure~\ref{fig:om3PWL}.

\begin{lema}
The heteroclinic bifurcation curve is
\[
F^{h}(x_e) = \frac{2\left(x_e - \frac{x_0 + x_1}{2} \right)}{(\sqrt{x_0 x_1} - x_e)(\sqrt{x_0 x_1} + x_e)}
\]
for $\sqrt{x_0 x_1} < x_e < \frac{x_0 + x_1}{2}$. This function is analytic, strictly decreasing, and satisfies:
\[
F^{h}\left( \frac{x_0 + x_1}{2} \right) = 0, \qquad
\dfrac{dF^{h}}{dx_{e}}\left( \frac{x_0 + x_1}{2} \right) = -\frac{8}{(x_1 - x_0)^2}.
\]
\end{lema}

\begin{pro}
For the PWL system \eqref{eq:pwlinear}, the heteroclinic bifurcation condition can 
be written either in the form $x_e^{h}(F)$ or in the form $ F^{h}(x_e)$.
\end{pro}
\begin{proof}
In both $\Sigma_1$ and $\Sigma_2$, the $x$-nullclines lie below the invariant manifolds. Computing the determinant that indicates the direction of rotation shows a consistent sign in both regions with respect to $x_e$ and $F$. Thus, both $x_{e}$ and $F$ are rotatory.
\end{proof}

The next result is proved in Appendix~\ref{sec:cyclicityheterocline}.
\begin{pro}\label{pro:huslc}
The heteroclinic connection $(\Gamma_{PWL})$ unfolds a unique hyperbolic stable limit cycle.
\end{pro}

\subsubsection*{No-Return Region.} The no-return region identifies the subset of the parameter space for which trajectories do not return to $\Sigma$ once they cross it.  This behaviour is the PWL analogue of the focus-to-node transition in the original smooth system~\eqref{eq:4}.  In the nonlinear model, a focus allows the flow to spiral and eventually return to the neighbourhood of the equilibrium, while a node does not generate such 
rotational behaviour.  Thus, the no-return region in the PWL system corresponds precisely to the parameter configurations in the original model where the coexistence equilibrium 
acts as a node: trajectories cross the section once and then move monotonically towards their limit set without producing recurrent or oscillatory motion.

\begin{pro}
The critical curves delimiting the no-return region are:
\begin{align*}
F_{B_1} = \frac{2(x_1 - x_0)(x_e - \frac{x_0 + x_1}{2})}{x_1(x_e - x_0)^2}\qquad \text{and} \qquad F_{B_2} = \frac{2(x_1 - x_0)(x_e - \frac{x_0 + x_1}{2})}{x_0(x_e - x_1)^2}.
\end{align*}
\begin{proof}
The proof follows just checking that $F_{B_1}$ (resp. $F_{B_2}$) represents the connection between the invariant manifold of $x_0$ (resp. $x_1$) and the nullcline of $\dot{x}_2$ (resp. $\dot{x}_1$).
\end{proof}
\end{pro}

\subsubsection{Description of the parameter space}
In this subsection we describe the parameter space $(x_e, F)$ of the PWL equation~\eqref{eq:pwlinear}, showing how the bifurcation curves previously characterised delimit the different dynamical regions. For each region we summarise the corresponding phase-space behaviour:
\begin{enumerate}[(a)]

    \item 
    $\displaystyle 
    \Omega_{1}
    =\left\{
        \tfrac{x_{0}+x_{1}}{2} < x_{e} < x_{1},
        \;\; F \le F_{B_{2}}
      \right\}
    $:
    All trajectories in the basin of attraction converge monotonically to the stable equilibrium, leading to stable coexistence of consumer and vegetation.

    \item 
    $\displaystyle 
    \Omega_{2}
    =\left\{
        \tfrac{x_{0}+x_{1}}{2} < x_{e} < x_{1},
        \;\; F > F_{B_{2}}
      \right\}
    $:
    Trajectories converge to the stable equilibrium through damped oscillations.

    \item 
    $\displaystyle 
    \Omega_{3}
    =\left\{
        (x_{e},F)\in(x_{0},x_{1})\times\mathbb{R}^{+}
        \;\middle|\;
        x_{e}=\tfrac{x_{0}+x_{1}}{2}
      \right\}
    $:
    This curve corresponds to the Hopf bifurcation and is the \emph{only} regime in which no sliding segment is present.  
    Trajectories in the basin converge to the equilibrium through damped oscillations.

    \item 
    $\displaystyle 
    \Omega_{4}
    =\left\{
        x_{0}< x_{e} < \tfrac{x_{0}+x_{1}}{2},
        \;\; x_{e} > x_{e}^{h}(F)
      \right\}
    $:
    All interior trajectories tend to a stable limit cycle, corresponding to sustained predator-prey oscillations.

    \item 
    $\displaystyle 
    \Omega_{5}
    =\left\{
        x_{0}< x_{e} < \tfrac{x_{0}+x_{1}}{2},
        \;\; x_{e} = x_{e}^{h}(F)
      \right\}
    $:
    The stable limit cycle collides with the boundary, producing a heteroclinic loop.  
    This configuration is structurally unstable: perturbations of $(x_{e},F)$ break the loop, yielding either a stable cycle or extinction.

    \item 
    $\displaystyle 
    \Omega_{6}
    =\left\{
        x_{0}< x_{e} < \tfrac{x_{0}+x_{1}}{2},
        \;\; x_{e} < x_{e}^{h}(F),
        \;\; F > F_{B_{1}}
      \right\}
    $:
    The limit cycle has disappeared, and trajectories exhibit oscillatory extinction with long extinction times.

    \item 
    $\displaystyle 
    \Omega_{7}
    =\left\{
        x_{0}< x_{e} < \tfrac{x_{0}+x_{1}}{2},
        \;\; F \le F_{B_{1}}
      \right\}
    $:
    Trajectories are directed monotonically toward the boundary equilibrium, corresponding to certain extinction of the consumer.
\end{enumerate}

In all regimes except $\Omega_{3}$, the switching boundary contains a sliding segment.  
When trajectories enter this region, the discontinuous nature of the vector field produces a sudden change in the consumer population: solutions are rapidly displaced along the boundary before re-entering the interior or approaching the extinction equilibrium.  
This sliding-induced jump is responsible for the abrupt transitions observed in several dynamical regimes.

\begin{figure}[h]
    \centering

    \begin{minipage}{1\textwidth}
        \centering
        \begin{subfigure}[t]{0.215\textwidth}
            \centering
            \begin{overpic}[width=\textwidth]{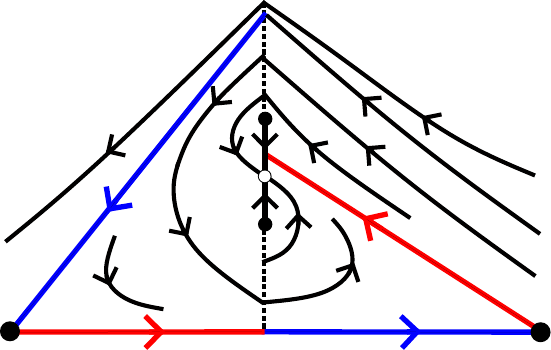}
            \end{overpic}
            \caption{$\Omega_{1}$.}
            \label{fig:om1PWL}
        \end{subfigure}
        \hspace{0.05\textwidth} 
        \begin{subfigure}[t]{0.215\textwidth}
            \centering
            \begin{overpic}[width=\textwidth]{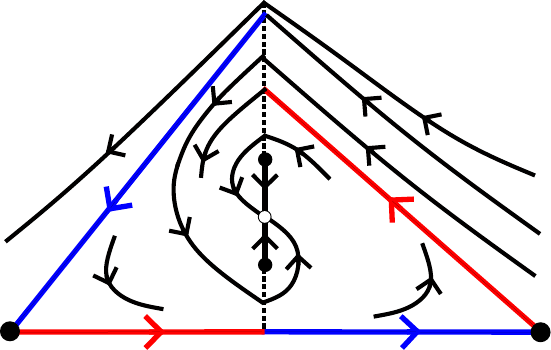}
            \end{overpic}
            \caption{$\Omega_{2}$.}
            \label{fig:om2PWL}
        \end{subfigure}
        \hspace{0.05\textwidth} 
        \begin{subfigure}[t]{0.215\textwidth}
            \centering
            \begin{overpic}[width=\textwidth]{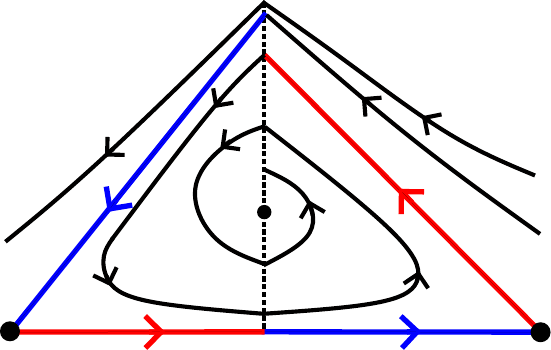}
            \end{overpic}
            \caption{$\Omega_{3}$.}
            \label{fig:om3PWL}
        \end{subfigure}
    \end{minipage}
    \vspace{2em}
    \begin{minipage}{1\textwidth}
        \centering
        \begin{subfigure}[t]{0.215\textwidth}
            \centering
            \begin{overpic}[width=\textwidth]{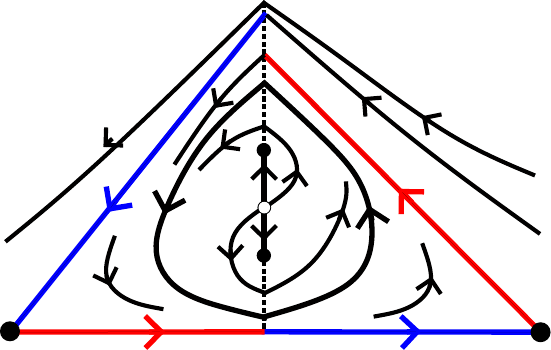}
            \end{overpic}
            \caption{$\Omega_{4}$.}
            \label{fig:om4PWL}
        \end{subfigure}
        \hspace{0.03\textwidth} 
        \begin{subfigure}[t]{0.215\textwidth}
            \centering
            \begin{overpic}[width=\textwidth]{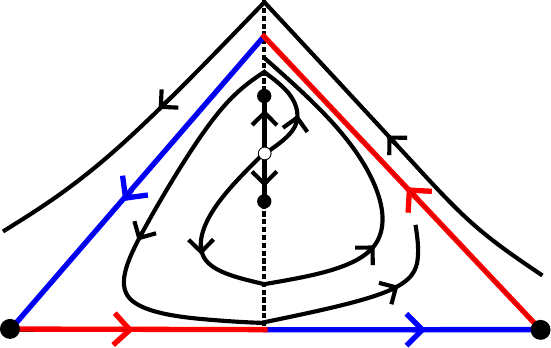}
            \end{overpic}
            \caption{$\Omega_{5}$.}
            \label{fig:om5PWL}
        \end{subfigure}
        \hspace{0.03\textwidth} 
        \begin{subfigure}[t]{0.215\textwidth}
            \centering
            \begin{overpic}[width=\textwidth]{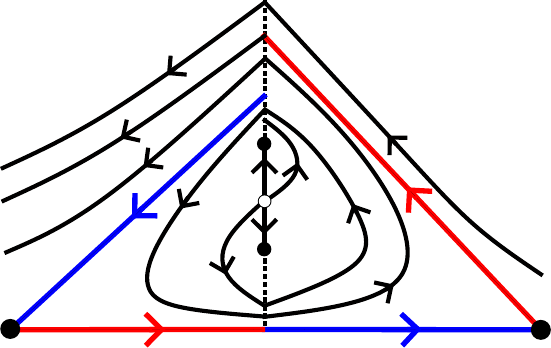}
            \end{overpic}
            \caption{$\Omega_{6}$.}
            \label{fig:om6PWL}
        \end{subfigure}
        \hspace{0.03\textwidth} 
        \begin{subfigure}[t]{0.215\textwidth}
            \centering
            \begin{overpic}[width=\textwidth]{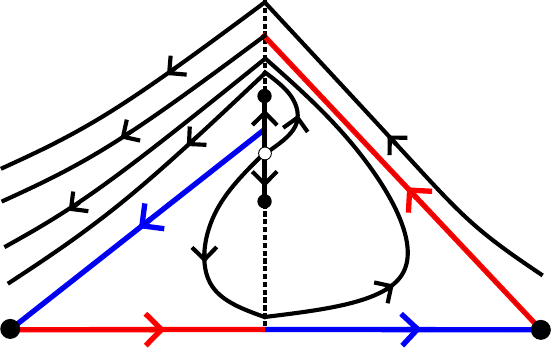}
            \end{overpic}
            \caption{$\Omega_{7}$.}
            \label{fig:om7PWL}
        \end{subfigure}
    \end{minipage}
    \captionsetup{width=\textwidth}
    \caption{Phase portraits of all the regions in the parameter space $(x_{e},F)$, where $F > 0$ and $x_0 < x_e < x_1$. Stable manifolds are coloured blue and unstable manifolds coloured in red.}
    \label{fig:pporiginal1}
\end{figure}

 \begin{figure}[h]
	\centering
	\begin{overpic}[width=0.4\textwidth]{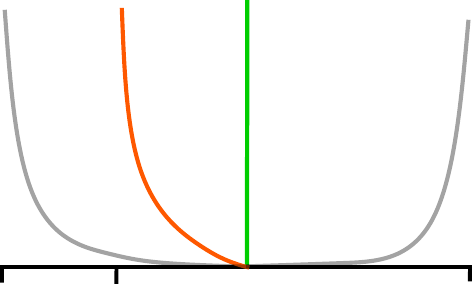}
		\put(-0.5,-4){$x_{0}$}
		\put(15,-5){$\sqrt{x_0x_1}$}
		\put(46,-4){$\frac{x_{0}+x_{1}}{2}$}
		\put(98,-4){$x_{1}$}
        \put(93,7){$\Omega_{1}$}
        \put(78,44){$\Omega_{2}$}
        \put(38,44){$\Omega_{4}$}
        \put(12,24){$\Omega_{6}$}
        \put(4,8){$\Omega_{7}$}
	\end{overpic}
	\vspace{0.3cm}
    	 \captionsetup{width=\textwidth}
	\caption{Parameter space $(x_{e},F)$. The orange curve of the form $x_{e}(F)$ corresponds to the heteroclinic bifurcation curve (region $\Omega_{5}$). The green curve corresponds to the pseudo-Hopf bifurcation (region $\Omega_{3}$). The grey curves separate the no return regions in the parameter space from the ones that have return. They are analogous to the focus-node transition curves separating the node and focus.}
\end{figure}
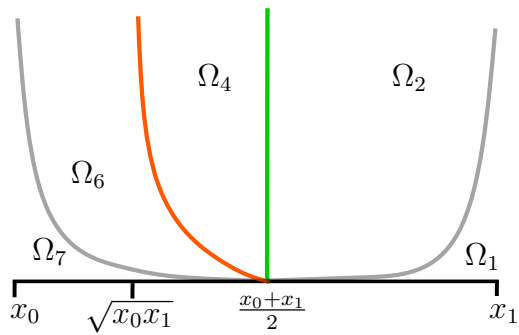
\subsubsection{Effect of habitat destruction in the heteroclinic bifurcation curve}
In this section, we show analytically that habitat destruction $D$ anticipates
the heteroclinic bifurcation when the consumer efficiency $F$ is fixed.  
To establish this fact, we compare the Hopf bifurcation curve $x_H$
with the heteroclinic bifurcation curve $x_e^{h}(F,D)$ and compute their horizontal
separation.  
We prove that this distance decreases monotonically with $D$, implying that habitat
loss pushes the heteroclinic connection closer to the Hopf threshold and therefore
shrinks the oscillatory region of the system.

\medskip
In the PWL model, equation \eqref{eq:pwlinear}, the heteroclinic bifurcation curve is given explicitly by
$$
F^{h}(x_e;D)
= \frac{2\,(x_e - x_H)}{x_0 x_1 - x_e^{\,2}},
\qquad
x_H = \frac{x_0+x_1}{2}=\frac{1-D}{2}.
$$
Solving $F^{h}(x_e;D)$ for $x_e$ yields
$$
x_e^{h}(F,D)
= \frac{-1 + \sqrt{\,1 + F^{2} x_0 x_1 + 2F\,x_H\,}}{F}.
$$
Since $x_0 x_1=\varepsilon/\alpha$ is independent of $D$, the dependence of
$x_e(F,D)$ on habitat destruction enters only through $x_H$.  
Differentiating with respect to $D$ at fixed $F$ gives
$$
\frac{\partial x_e}{\partial D}
= -\frac{1}{2 \sqrt{\,1 + F^{2} x_0 x_1 + 2F\,x_H\,}}
< 0,
$$
showing that the heteroclinic curve moves leftwards as $D$ increases.

\medskip
The pseudo-Hopf bifurcation occurs at
$x_H=\frac{1-D}{2},$
which also shifts left with increasing habitat loss.  
The horizontal distance between the pseudo-Hopf and the heteroclinic curves is
$$
d_{o}(D)=x_e^{(\mathrm{H})}-x_e(F,D).
$$
Differentiating with respect to $D$ yields
$$
d_{o}'(D)
= -\frac{d x_H}{dD}-\frac{\partial x_e}{\partial D}
= -\frac{1}{2}
+ \frac{1}{2\sqrt{\,1 + F^{2} x_0 x_1 + 2F\,x_H\,}}
<0.
$$

\medskip
Therefore, both bifurcation curves drift towards smaller values of $x_e$ as habitat
destruction increases, but the pseudo-Hopf curve does so more rapidly.  
Consequently, the distance $d_{o}(D)$ decreases strictly with $D$, confirming
explicitly that habitat loss anticipates the heteroclinic bifurcation and reduces
the parameter region in which oscillations can occur.  
This provides a rigorous analytical explanation for the loss of resilience observed
in the PWL model.

\subsubsection{Resilience of the oscillatory regime}
In the PW model, the heteroclinic bifurcation curve is known explicitly. 
This allows for a precise quantification of resilience throughout the entire one-limit-cycle region in parameter space. 
For a given parameter pair $(x_e, F)$ in this region, the distance to the heteroclinic curve can be computed by minimizing
$$
d = \sqrt{(x_e - x_e^*)^2 + (F - F^*)^2},
$$
where $(x_e^*, F^*) \in F_c(x_e)$ lies on the heteroclinic bifurcation curve.

This distance provides a natural measure of resilience: it quantifies how much a parameter pair can be perturbed before crossing the heteroclinic curve and losing the limit cycle. In other words, it identifies the smallest change in parameter space capable of pushing the system into the collapse regime. The same idea could, in principle, be applied to smooth systems, provided the heteroclinic bifurcation curve is explicitly known or sufficiently well approximated.

\subsubsection{Similarities with the original model}
The piecewise linear (PWL) model and the original system share the same qualitative bifurcation structure: both exhibit a Hopf bifurcation and a heteroclinic bifurcation. However, their quantitative properties differ, particularly in the slope of the heteroclinic bifurcation curve at the Hopf point and in their asymptotic behaviour. While the curve is analytic throughout the parameter space in the PWL model, it loses analyticity in the original system precisely at the point where the heteroclinic connection intersects the Hopf curve.

In both models, the hyperbolicity ratio along the heteroclinic orbit matches. This quantity determines whether the connection is contractive or expansive and restricts how limit cycles can unfold. As in the original system, the heteroclinic connection of the PWL model admits only a single stable limit cycle (see Proposition~\ref{pro:huslc}). Furthermore, the pseudo-Hopf bifurcation in the PWL system unfolds a unique crossing limit cycle, just as the classical Hopf bifurcation in the smooth model gives rise to a single limit cycle.

In both models we proved that the heteroclinic bifurcation curve can be expressed as $F^{h}(x_{e})$, which is a monotonic decreasing function. However, their asymptotic limits differ significantly: in the PWL model, the heteroclinic curve approaches the geometric mean $x_e = \sqrt{x_0 x_1}$, whereas in the smooth system it approaches the harmonic mean $x_e = \tfrac{2 x_0 x_1}{x_0 + x_1}$.

\subsection{Impact of extrinsic noise} \label{sec:Stoch}
In this section we numerically explore the impact of extrinsic noise in the overall dynamics of the system. To do so, we consider the stochastic version of system \eqref{eq:4} by introducing additive noise into the resource species:
\begin{equation*}
	\begin{aligned}
		\dot{x} &= -\frac{x^3}{x_{0} x_{1}} + \frac{(x_{0} + x_{1}) x^2}{x_{0} x_{1}} - x y - x + \sigma \eta(t), \\
		\dot{y} &= F x y - F x_{e} y.
	\end{aligned}
\end{equation*}
The system now includes a stochastic perturbation in the equation  $\dot{x}$ modeling environmental variability, making the resource species fluctuate randomly. Here, $\eta(t)$ denotes a standard white Gaussian noise process.

In the discrete-time numerical implementation, $\eta(t)$ is approximated by independent samples drawn from the normal distribution $\mathcal{N}(0, \Delta t)$, where $\Delta t$ is the time step used in the integration scheme. Note that if $\eta \sim \mathcal{N}(0, \Delta t)$, then the noise term $\sigma \eta \sim \mathcal{N}(0, \sigma^2 \Delta t)$, so the parameter $\sigma$ directly controls the \emph{noise intensity} in the system.

This stochastic extension allows us to study the robustness of the bifurcation structure under random perturbations and to investigate noise-induced extinction phenomena.

We fix the parameters \( x_0 = 1 \) and \( x_1 = 3 \) for all simulations, while \( x_e \) and \( F \) vary between simulations. The time step is set to \( \Delta t = 0.01 \), and the simulation time spans the interval \( t \in [0, 300] \).

Initial conditions are chosen as \( (1.5, 0.3) \), which lie near the basin of attraction of the limit cycle emerging from \( (x_e, y_e) \). This ensures that forward time simulations either tend to the origin or approach the limit cycle (when it exists), given enough time. This setup is appropriate for studying the global bifurcation behaviour under noise, since the limit cycle ultimately collides with the heteroclinic orbit during bifurcation.

In simulations of survival probability and extinction times, each pixel in the resulting plots aggregates data from \( n = 90 \) independent realizations. To improve the accuracy of
results, one can increase either the number of simulations per pixel or the total simulation
time, since in some cases, the system may be mistakenly classified as non-extinct, while
it would in fact go extinct as $t$ goes to $\infty$.

Ideally, a more exhaustive study would explore variations in all parameters $x_0, x_1, x_e,$ and $F$. However, due to the high dimensionality of the parameter space, data visualization becomes challenging. Therefore, the numerical simulations presented here serve as illustrative examples that show what types of behaviour can occur under certain conditions, but they do not rule out the possibility of other behaviours under different parameter values or initial conditions.

\subsubsection{Anticipation of heteroclinic bifurcation under the influence of noise}
The extinction situation occurs in the case where the limit cycle has already collapsed with the heteroclinic orbit or when the orbit scapes the outer basin of attraction of the limit cycle. So basically a probability of extinction different from zero what is telling us is that there is some scenario in which noise can anticipate the extinction. For the deterministic case, in which $\sigma=0$, the heteroclinic curve occurs at one point and the survival probability transitions abruptly from 1 to 0. When considering $\sigma>0$ the averages of survival probability transition smoothly at first glance, meaning that the heteroclinic bifurcation occurs with certain probability. When noise is present, this transition becomes gradual: the averaged survival probability decays smoothly, suggesting that the heteroclinic bifurcation is anticipated with some probability, depending on the noise intensity. Consequently, the optimal threshold $x_c$ introduced in Section~\ref{subsec:habitatdestruction} would also be anticipated by noise, just as the heteroclinic bifurcation is. Although $x_c$ depends on the level of habitat destruction $D$ and the analysis could, in principle, be carried out directly in terms of $D$, working with $x_e$ is more convenient. Expressing the system in terms of $D$ would require a nonlinear rescaling that distorts the geometry of the parameter space, whereas the $(x_e, F)$ representation preserves a simpler structure without altering the underlying mechanisms or qualitative conclusions.
    
\begin{figure}[h]
    \centering
    \begin{subfigure}[t]{0.48\textwidth}
        \centering
        \includegraphics[width=\textwidth]{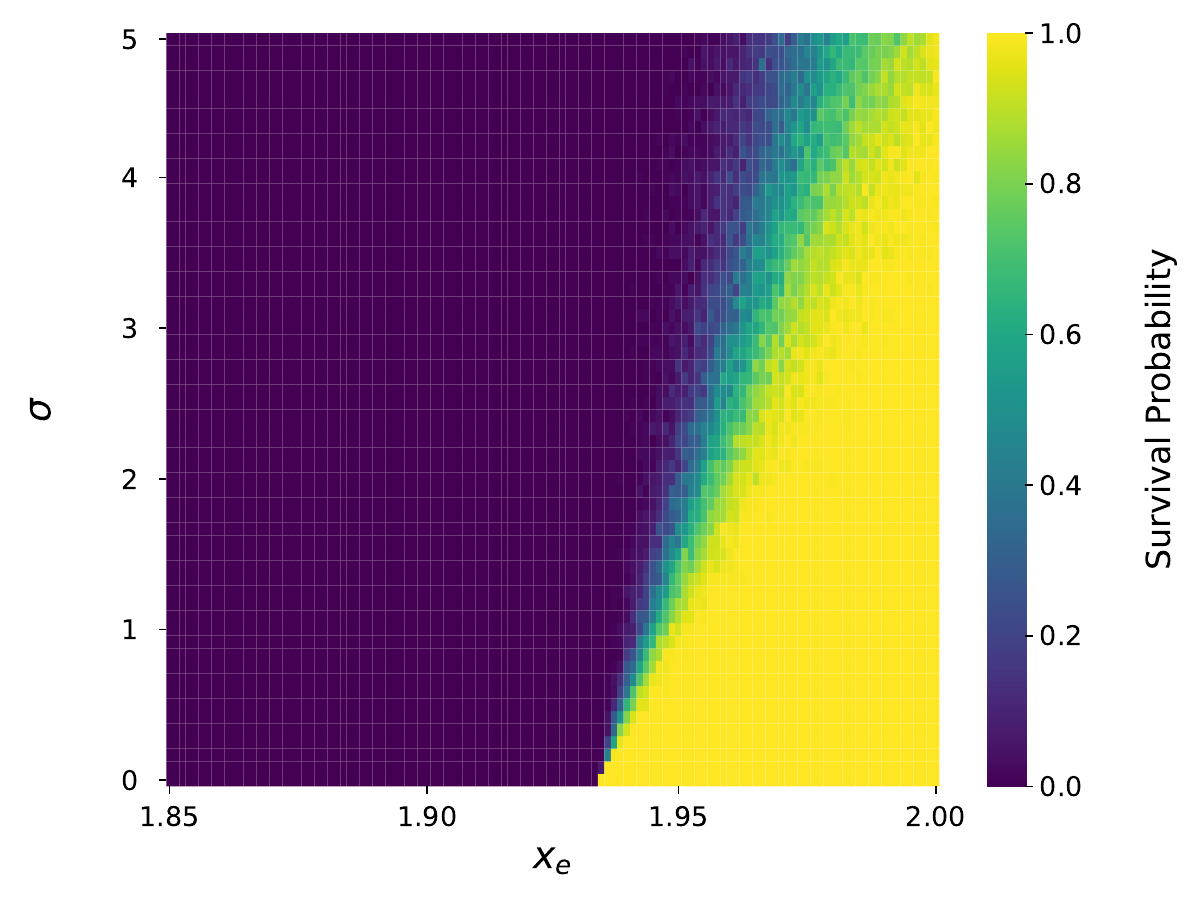}
        \caption{$F=0.5$}
        \label{fig:sbif05}
    \end{subfigure}
    \hfill
    \begin{subfigure}[t]{0.48\textwidth}
        \centering
        \includegraphics[width=\textwidth]{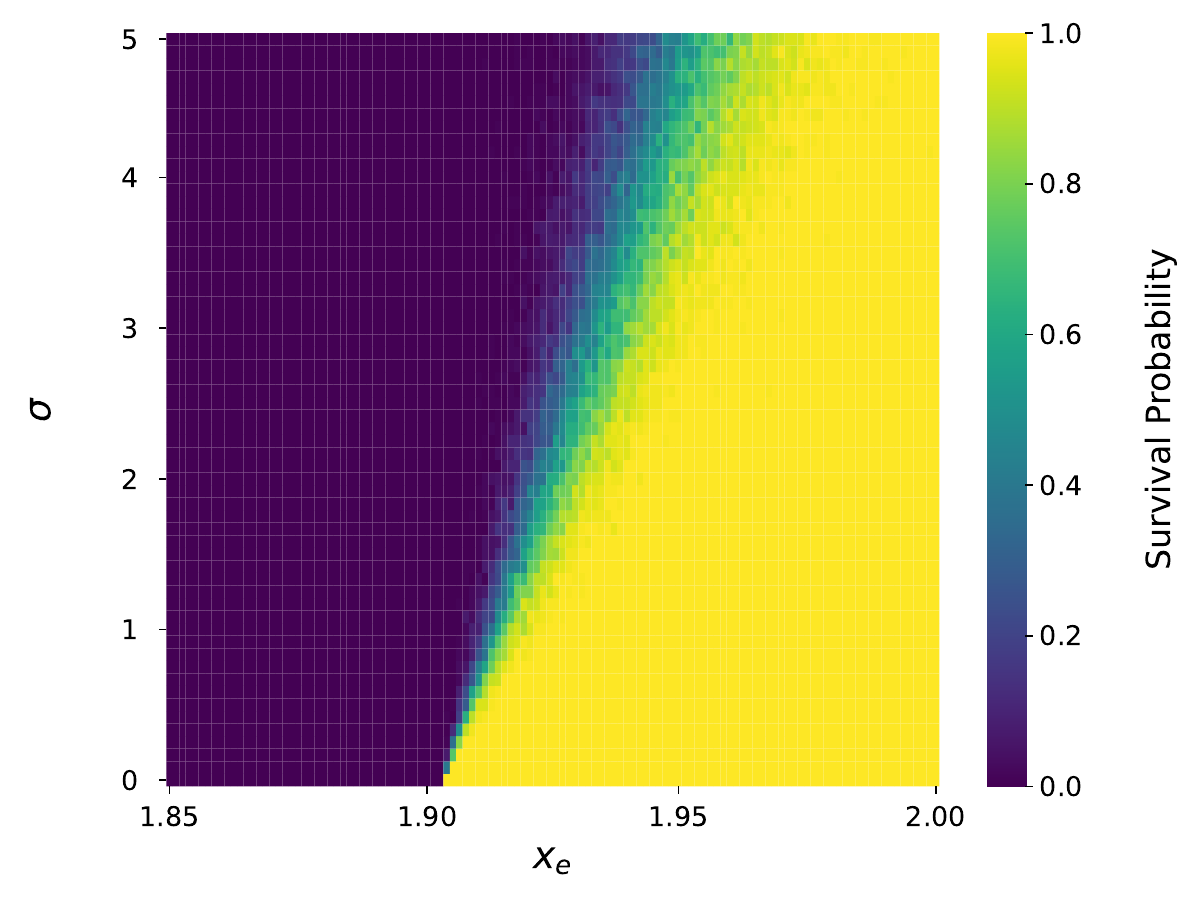}
        \caption{$F=1$}
        \label{fig:sbif1}
    \end{subfigure}
    \vspace{0.5cm}
    \begin{subfigure}[t]{0.48\textwidth}
        \centering
        \includegraphics[width=\textwidth]{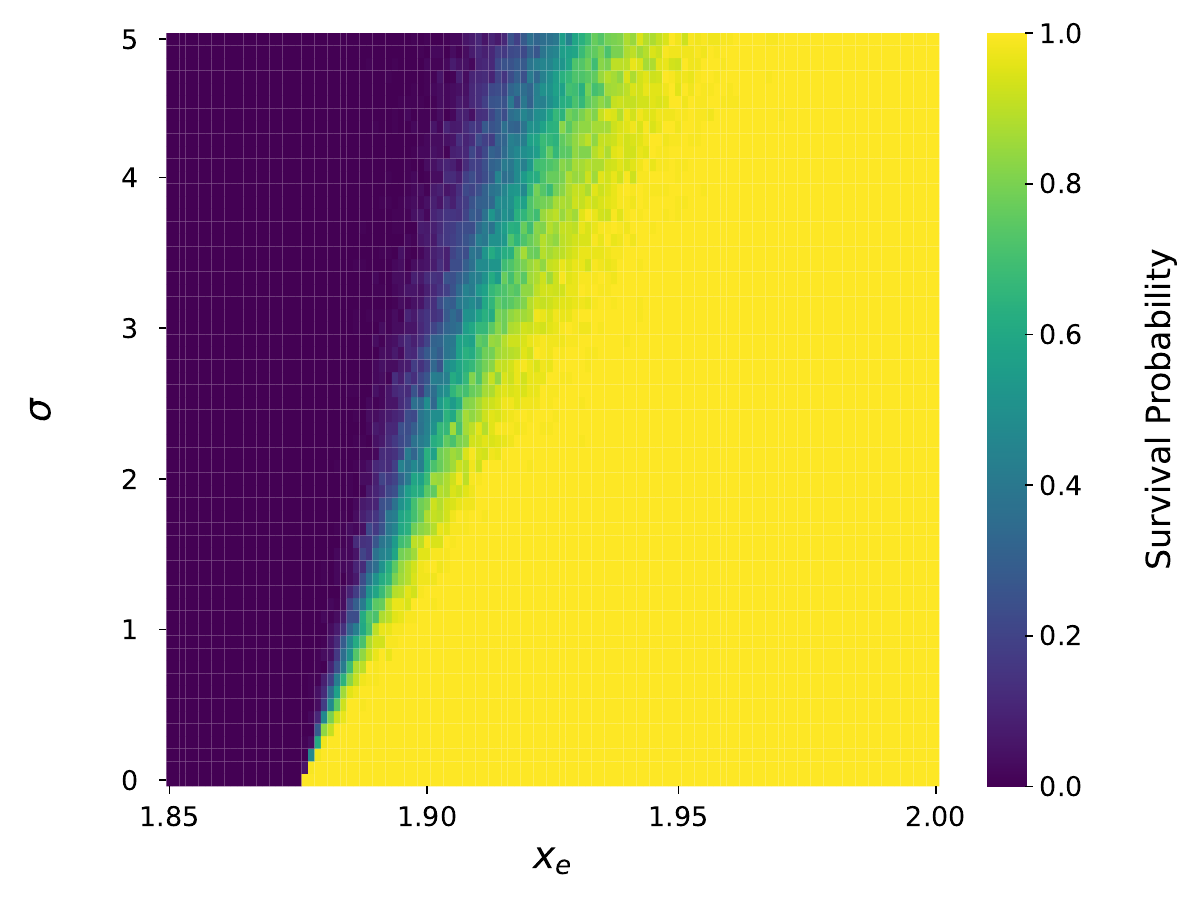}
        \caption{$F=2$}
        \label{fig:sbif2}
    \end{subfigure}
    \hfill
    \begin{subfigure}[t]{0.48\textwidth}
        \centering
        \includegraphics[width=\textwidth]{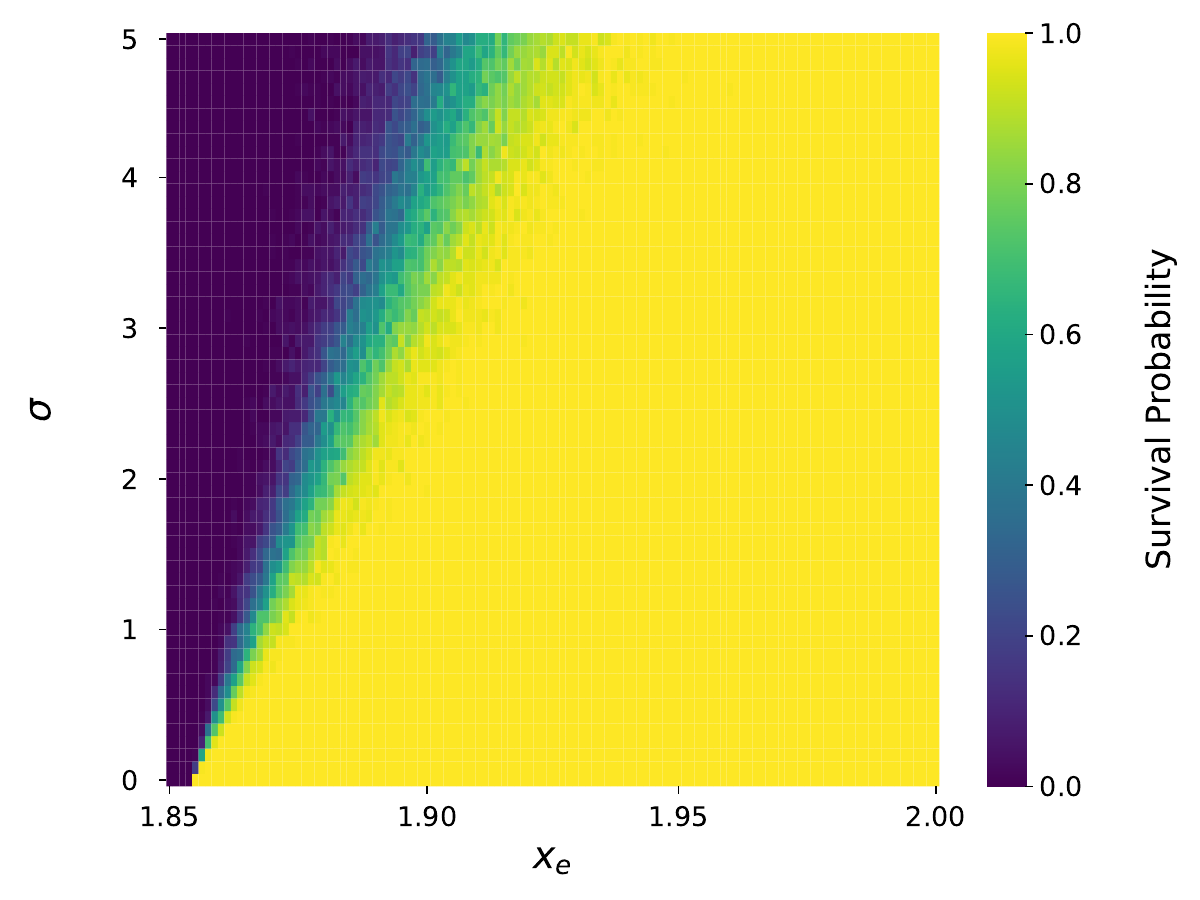}
        \caption{$F=5$}
        \label{fig:sbif5}
    \end{subfigure}
   	 \captionsetup{width=\textwidth}
    \caption{Survival probability fixing a certain parameter value $F$ over $x_e$ when increasing the noise intensity $\sigma$. The transition of survival probability from one to zero for $\sigma=0$ corresponds to the heteroclinic bifurcation. So for $\sigma>0$, there is that the bifurcation anticipates. The stability region between the Hopf and the heteroclinic bifurcation corresponds to the vertical length of $x_{e}$, given an $F$ and a certain noise intensity ($\sigma$). } 
    \label{fig:survprob}
\end{figure}

 In Figure~\ref{fig:survprob}, it is seen that the grater the noise intensity the more it anticipates the global bifurcation. The yellow region stands for the region in the parameter space $(\sigma,x_{e})$ in which the populations persists during the simulation time interval. Take into account that the stability region does not increase indefinitely since there is a bound at which there is no heteroclinic orbit. But maybe would be interesting to compute the slope of the yellow region to see if when increasing $F$ it increases the stability region.

\subsubsection{Regimes under the influence of noise and extinction times}

We focus on the region of parameter space where $x_e < (x_0 + x_1)/2$, which corresponds to the zone where oscillatory behaviour and collapse can occur.

It is particularly interesting to analyse how extinction times change when we increase the noise intensity, keeping the parameters $x_0$, $x_1$, and $F$ fixed, and decreasing the value of $x_e$, which can be interpreted as a measure of how degraded or vulnerable the community is. Qualitatively, the system exhibits two dynamical regimes: persistence or extinction. Notably, the presence of noise can induce complex oscillations that either delay extinction or anticipate the collapse.

\begin{figure}[h]
    \centering
    \begin{subfigure}[t]{0.45\textwidth}
        \centering
        \includegraphics[width=\textwidth]{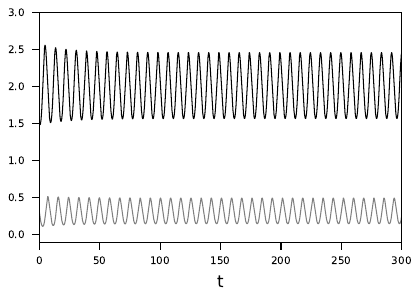}
        \caption{$x_e = 1.98$, $\sigma = 0$}
        \label{fig:ts1}
    \end{subfigure}
    \hfill
    \begin{subfigure}[t]{0.45\textwidth}
        \centering
        \includegraphics[width=\textwidth]{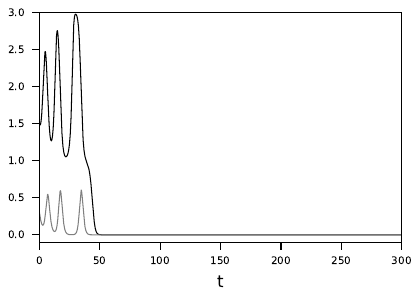}
        \caption{$x_e = 1.90$, $\sigma = 0$}
        \label{fig:ts2}
    \end{subfigure}
    \vspace{0.5cm}
    \begin{subfigure}[t]{0.45\textwidth}
        \centering
        \includegraphics[width=\textwidth]{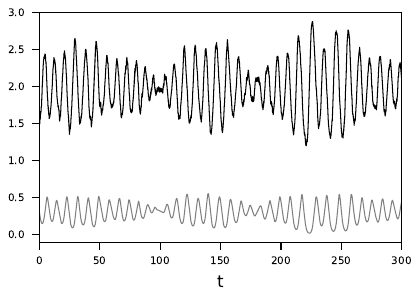}
        \caption{$x_e = 1.98$, $\sigma = 5$}
        \label{fig:ts3}
    \end{subfigure}
    \hfill
    \begin{subfigure}[t]{0.45\textwidth}
        \centering
        \includegraphics[width=\textwidth]{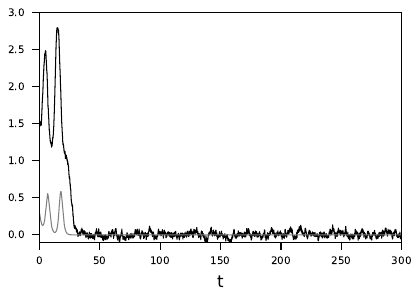}
        \caption{$x_e = 1.90$, $\sigma = 5$}
        \label{fig:ts4}
    \end{subfigure}
    \captionsetup{width=\textwidth}
    \caption{Time series showing the effect of noise on extinction and persistence. Fixed $F=1$; we vary both $x_e$ and $\sigma$. The choice of $x_e$ distinguishes between values near the Hopf point and values near the heteroclinic bifurcation. Black denotes $x(t)$, and grey denotes $y(t)$.}
    \label{fig:timeextinction}
\end{figure}
Figure~\ref{fig:ts1} displays periodic behaviour typical of deterministic oscillations in the limit cycle regime, while Figure~\ref{fig:ts3} shows aperiodic dynamics induced by noise. This contrast highlights how noise disrupts the periodic structure of the deterministic system.

Figure~\ref{fig:ts4} illustrates a realization where the extinction time is strongly influenced by stochasticity. Depending on the trajectory, extinction may occur sooner or later, demonstrating that noise can either push the system outside the basin of attraction or temporarily sustain it within. This situation corresponds to the intermediate survival regime $0 < P_s < 1$ observed in Figure~\ref{fig:survprob}. Additionally, there are instances in which noise induces chaotic-like oscillations that delay collapse, effectively increasing the extinction time.

\begin{figure}[H]
	\centering
	\includegraphics[width=0.9\textwidth]{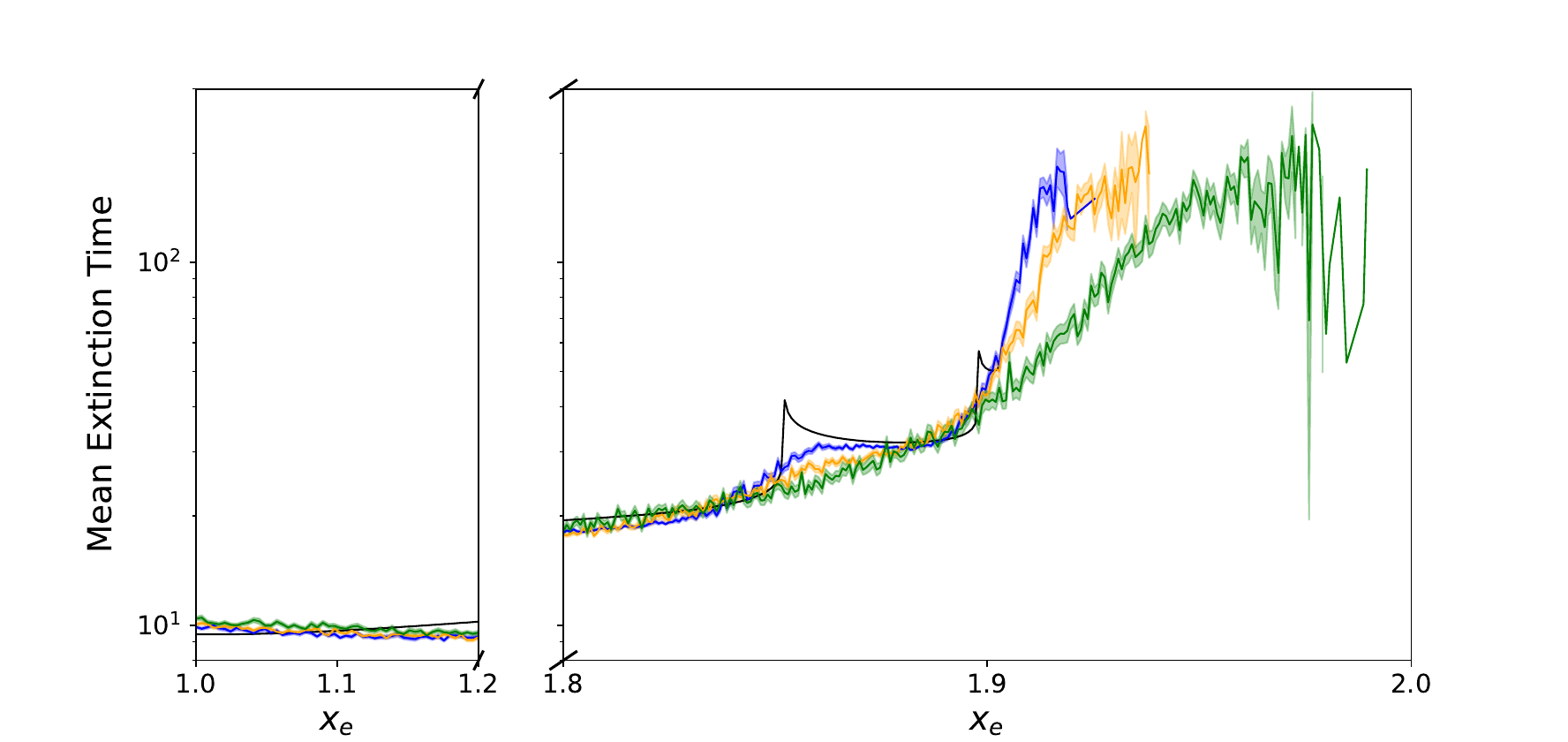}
	\captionsetup{width=\textwidth}
    \caption{Extinction times for fixed parameters $x_0 = 1$, $x_1 = 3$, and $F = 1$, with $x_e \in [x_1, (x_0 + x_1)/2]$ shown in linear-log scale. Each curve corresponds to a different noise intensity: green for $\sigma = 5$, yellow for $\sigma = 2$, blue for $\sigma = 1$, and black for $\sigma = 0$. Solid lines represent the mean extinction time, and the shaded areas represent the standard deviation. Each time series ends at the heteroclinic bifurcation point, beyond which extinction no longer occurs.}
\end{figure}

For values of $x_e$ near the Hopf bifurcation and under high noise intensity, the limit cycle is both small and surrounded by a narrow basin of attraction. In such cases, stochastic fluctuations can rapidly push trajectories out of the basin, leading to quick extinction.

As shown in the extinction time curves, increasing noise intensity generally reduces the average extinction time. However, in certain regions of parameter space, noise can also delay extinction by inducing chaotic-like oscillations that momentarily trap the system within the basin of attraction. This effect is reflected in the non-monotonic trends and wider variability in extinction times for intermediate values of $x_e$.

\section{Conclusions}

This study addresses a central ecological problem: how habitat loss affects the persistence of consumer-resource communities when facilitation among resource individuals is present. Such positive interactions--characteristic of ecosystems like drylands, where vegetation facilitation shapes spatial structure and functional resilience~\cite{Berdugo2017}--introduce nonlinear responses and can trigger sudden regime shifts, making ecological collapses difficult to predict.

A recent work has examined the effects of facilitation and habitat destruction, but the analysis was mostly restricted to local bifurcations, with global behaviour inferred primarily from numerical simulations~\cite{JosVidBlaiErnest2021}. Although those studies revealed long transients, relaxation oscillations, and an eventual collapse caused by a global bifurcation, the heteroclinic mechanism driving this collapse was not analytically proven. In this work, we provide a comprehensive analysis of this system focusing on the role of habitat destruction, showing how it alters the global dynamics, reduces the coexistence region, and pushes the system toward the heteroclinic bifurcation that ultimately anticipates the collapse threshold.

To address this issue, we first provide an analytical description of the global phenomena highlighted in \cite{JosVidBlaiErnest2021}, using tools such as rotated vector fields, Poincar\'e compactification, and slow-fast theory. A key novelty of this study is the introduction of a piecewise-linear (PWL) approximation, which preserves the qualitative behaviour of the original model while offering complete analytical control over the global bifurcation. In this PWL framework, the heteroclinic bifurcation curve can be computed explicitly, providing a powerful tool for conservation applications by quantifying how far a given system lies from a collapse-inducing change in parameters.

The model can thus be regarded as a simple yet insightful example for studying global phenomena such as heteroclinic connections in ecological systems. Analytical studies of such global bifurcations are relatively rare, since most ecological models rely on numerical exploration. An exception is provided in \cite{heteroclinicConnectionPredatorPrey}, where the authors, unable to track invariant manifolds directly, transform the system into a near-integrable form to describe the corresponding heteroclinic connection.

Finally, we also investigate how extrinsic noise affects these global dynamics. Our results show that stochastic fluctuations can make the heteroclinic bifurcation occur earlier than predicted by deterministic theory, effectively shrinking the region of persistent oscillations and increasing the probability of collapse. This demonstrates that environmental variability can further reduce ecosystem resilience, emphasizing the importance of considering stochastic effects when evaluating collapse thresholds.

A promising direction for future research is to build systematic bridges between classical smooth bifurcations and their piecewise-linear (PWL) counterparts. Developing such correspondences would allow the construction of PWL models that retain the essential qualitative behaviour of smooth systems while being far more analytically tractable. This is particularly relevant in conservation ecology, where predicting and managing collapse requires robust analytical tools: smooth models often demand solving Hamilton-Jacobi equations numerically or rely heavily on computational exploration, whereas PWL systems provide explicit formulas and complete analytical control. Establishing these links could therefore enhance our ability to design effective interventions and anticipate critical transitions with greater precision.

\section{Acknowledgements} This research is funded by the 2020-21 Biodiversa+/Water JPI joint call under the BiodivRestore ERANET Cofund (GA N.101003777) MPA4Sustainability project with funding organizations: Innovation Fund Denmark (IFD), Agence Nationale de la Recherche (ANR), Fundaçao para a Ciencia e a Tecnologia (FCT), Swedish Environmental Protection Agency (SEPA) and grant PCI2022-132926 (MCIN/AEI/10.13039/501100011033) and
by the European Union Next Generation EU/PRTR. This work is also
supported by the Spanish State Research Agency, through the Severo
Ochoa and María de Maeztu Program for Centers and Units of Excellence in R\&D (CEX2020-001084-M). We thank CERCA Programme/ Generalitat de Catalunya for institutional support. This research is also supported by the Spanish Government (Project PID2022-136613NB-I00). Helpful suggestions have been made by A. Gasull, A. Zegeling, and A. Teruel. 
\bibliographystyle{abbrv}
\bibliography{biblio.bib}
\appendix
\addtocontents{toc}{\protect\setcounter{tocdepth}{1}}
\section{Poincar\'e compactification}\label{sec:PC}
The Poincar\'e compactification is a common tool for studying the limit sets of planar polynomial differential systems. To carry out this compactification, the stereographic projection is used to embed the planar differential system onto the unit sphere. Under this map, the northern hemisphere becomes a compact domain whose equator represents points at infinity. Unfortunately, the vector field does not stay bounded as it approaches $\mathbb{S}^{1}$ and must be multiplied by a regularization function to smoothly extend the vector to infinity. 

By embedding the phase portrait into a compact space, we can apply Poincar\'e--Bendix\-son Theorem to characterize the limit sets of all orbits. For more details on the construction of the compactification and the dynamics near infinity, we recommend the works \cite{Andronov1966, DumLliArt2006}.

The vector fields in each of the three local charts are given below:

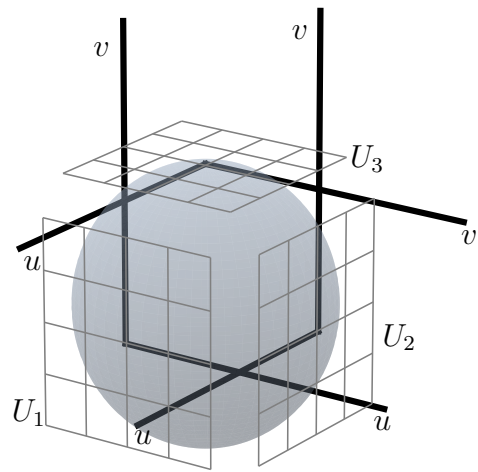
\begin{figure}[h]
  \centering
  \begin{minipage}[t]{0.48\textwidth}\vspace{0pt}
    \begin{itemize}
        \item  $(U_1,\phi_1)$:
        \begin{equation*}
          \begin{cases}
            \dot u = v^d\bigl[-u\,P(\tfrac1v,\tfrac uv)+Q(\tfrac1v,\tfrac uv)\bigr],\\[1ex]
            \dot v = -v^{d+1}P(\tfrac1v,\tfrac uv),
          \end{cases}
        \end{equation*}
        
        \vspace{1cm}
        
        \item  $(U_2,\phi_2)$:
    \begin{equation*}
          \begin{cases}
        \dot u = v^d\bigl[P(\tfrac uv,\tfrac1v)-u\,Q(\tfrac uv,\tfrac1v)\bigr],\\[1ex]
        \dot v = -v^{d+1}Q(\tfrac uv,\tfrac1v),
      \end{cases}
        \end{equation*}
               
        \vspace{1cm}
        
        \item  $(U_3,\phi_3)$:
        \begin{equation*}
        \begin{cases}
        \dot u = P(u,v),\\[1ex]
        \dot v = Q(u,v).
      \end{cases}
        \end{equation*}
    \end{itemize}
  \end{minipage}\hfill
  \begin{minipage}[t]{0.48\textwidth}\vspace{0pt}
    \centering
    \begin{overpic}[width=\textwidth]{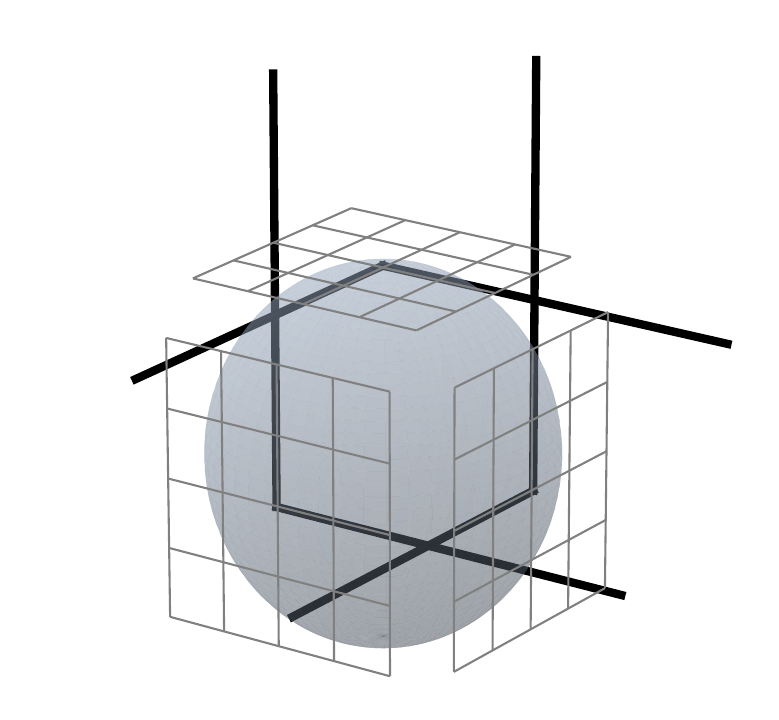}
      \put(16,14){$U_{1}$}
      \put(79.5,27){$U_{2}$}
      \put(74,57.8){$U_{3}$}
      \put(37,10){$u$}
      \put(78,12.5){$u$}
      \put(93,44.5){$v$}
      \put(64,80){$v$}
      \put(30,77){$v$}
      \put(18,40){$u$}
    \end{overpic}
    \captionof{figure}{Local charts $(U_k,\phi_k)$, $k=1,2,3$, of la esfera de Poincar\'e.}
    \label{fig:charts_compact}
  \end{minipage}
\end{figure}
Here, $d = \max\{\deg(P), \deg(Q)\}$. Each chart corresponds to a local parametrization of the vector field in the sphere, allowing us to describe the dynamics near the equator (i.e., at infinity in the plane). Local charts $(U_1,\phi_1)$ and $(U_2,\phi_2)$ study the points at infinity.

\begin{definition}[Finite singular points]
Singular points lying in \(\mathbb{S}^{2}\setminus\mathbb{S}^{1}\).
\end{definition}

\begin{definition}[Infinite singular points]
Singular points lying on \(\mathbb{S}^{1}\).
\end{definition}

When analysing singular points at infinity, we follow the same steps as for finite singularities, applying the theory of elementary and non-elementary singularities via linearization and blow-ups respectively. The Poincar\'e compactification transforms the original vector field into one that is homeomorphic to it, but not topologically conjugate. As a result, we can determine the qualitative nature of each singularity on the sphere at infinity, but we must use other methods to study quantitative features at infinity, such as eigenvalues or periods of limit cycles at infinity.

\section{Rotated vector fields}\label{sec:rvf}
The theory of rotated vector fields was first introduced by Duff in \cite{Duff1953LimitCyclesAR} for a particular family of vector fields called complete, and was later weakened into a wider family--the semicomplete vector fields--introduced in \cite{PERKO197563}. This theory provides a framework to track the movement of global objects such as limit cycles and invariant manifolds.

Consider a vector field
\begin{equation}\label{eq:rotated_vector_field}
\begin{aligned}
\dot{x} &= P(x,y,\nu), \\
\dot{y} &= Q(x,y,\nu),
\end{aligned}
\end{equation}
with a parameter $\nu \in (-\infty,\infty)$ and where $P$ and $Q$ are analytic functions. The family of vector fields \eqref{eq:rotated_vector_field} is called semicomplete if
\vspace{0.3cm}
\begin{itemize}
\item[(i)] their rest points remain fixed for all $\nu \in (-\infty,\infty)$,
\vspace{0.3cm}
\item[(ii)] $\dfrac{\partial \theta}{\partial \nu} = \dfrac{\partial}{\partial\nu} \arctan \dfrac{P}{Q} = \dfrac{P(\partial Q/\partial \nu) - Q(\partial P/\partial \nu)}{P^2 + Q^2} > 0$ for all $\nu \in (-\infty,\infty)$,

\vspace{0.3cm}
\item[(iii)] $\tan\theta \rightarrow \pm\infty$ as $\nu \rightarrow \pm\infty$, where $\theta = \tan^{-1}(Q/P)$ is the angle of the field vector $(P,Q)$.
\end{itemize}
\vspace{0.3cm}

In practice, these conditions are quite restrictive and are typically satisfied only within a certain region of the phase space. This family of vector fields ensures the monotonic movement of global objects such as limit cycles and invariant manifolds with respect to the parameter $\nu$, which will be referred to as the rotatory parameter. 

Condition (ii) is equivalent to studying the sign preservation of the determinant
\begin{equation*}
\Theta_{\nu}(x,y,\nu) =
det\begin{pmatrix}
P(x,y,\nu) & Q(x,y,\nu) \\[5pt]
\dfrac{\partial P(x,y,\nu)}{\partial \nu} & \dfrac{\partial Q(x,y,\nu)}{\partial \nu}
\end{pmatrix}.
\end{equation*}

Depending on the sign, the vector fields will rotate in one direction or the other: if positive, the rotation is counter-clockwise; if negative, it is clockwise. For families \eqref{eq:rotated_vector_field} satisfying the above conditions, theorems ensure the non-intersection of limit cycles and the monotonic growth of these with respect to the parameter. Further results are summarized in \cite{DumLliArt2006}; for more details, one may consult the original papers.
\section{Optimal bound via the singular heteroclinic connection}\label{sec:optimalbound}
We consider the two-dimensional vector field
$$
\dot x = P(x,y), 
\qquad 
\dot y = F\,y(x-x_{e}),
$$
with $F\gg 1$, and assume that 
$$
P(x,y)=P(x,0)-xy,
\qquad 
P(x,0)= -\,x\,\frac{(x-x_{1})(x-x_{0})}{x_{0}x_{1}}.
$$
For $x\in(x_{0},x_{1})$ one has $P(x,0)>0$, so the slow drift in $x$ is positive along the invariant axis $y=0$.

\medskip

When $F$ goes to $\infty$ the dynamics decomposes into fast vertical motions and a slow horizontal drift at large height.  
This motivates the singular heteroclinic approximation, in which we follow the flow along the contour
$$
\Gamma_{1}: (x_{1},y),\ y\in(\varrho,R), 
\qquad
\Gamma_{2}: (x,y),\ y=R,\ x\in[x_{1},x_{0}],
\qquad
\Gamma_{3}: (x_{0},y),\ y\in(R,\varrho),
$$
with $0<\varrho\ll1\ll R$.
Along this contour we compute the signed variation of $x$ using
$$
\frac{dy}{dx}
= \frac{Q(x,y)}{P(x,y)}
= \frac{F y (x-x_{e})}{P(x,y)}.
$$

To determine the singular heteroclinic connection, we evaluate the contributions along each segment.
Along the vertical line $x=x_{1}$ ($\Gamma_{1}$) we have $dx=0$, hence
$$
0 = \int_{x_{1}}^{x_{1}} dx
  = \int_{\varrho}^{R}\frac{P(x_{1},y)}{Q(x_{1},y)}\,dy .
$$
Since $P(x_{1},y)=-x_{1}y$ and $Q(x_{1},y)=F y (x_{1}-x_{e})$, we obtain
$$
\delta_{1}
= \int_{\Gamma_{1}}\frac{P}{Q}\,dy
= -\,\frac{x_{1}}{F(x_{1}-x_{e})}\,(R-\varrho).
$$

A symmetric computation along the vertical line $x=x_{0}$ ($\Gamma_{3}$) gives
$$
\delta_{3}
= \int_{\Gamma_{3}}\frac{P}{Q}\,dy
= \frac{x_{0}}{F(x_{e}-x_{0})}\,(R-\varrho).
$$

Along $\Gamma_{2}$, where $y=R$, one has
$$
\dot x = P(x,R)
       = P(x,0) - xR
       = -\,xR + \mathcal{O}(1),
$$
when $R$ goes to infinity, and therefore
$$
\delta_{2}
= \int_{\Gamma_{2}}\frac{P(x,R)}{Q(x,R)}\,dx
= O\!\left(\frac{1}{F}\right),
$$
a negligible contribution compared to $\delta_{1},\delta_{3}\sim R/F$.

\medskip

The singular heteroclinic closes when the leading vertical contributions cancel,
$$
\delta_{1}+\delta_{3}=0 .
$$
Substituting the expressions above gives
$$
\frac{x_{1}}{x_{1}-x_{e}}
=
\frac{x_{0}}{x_{e}-x_{0}}
\qquad\Longleftrightarrow\qquad
x_{e}
= \frac{2 x_{0}x_{1}}{x_{0}+x_{1}} .
$$

\begin{pro}[Optimal singular value]
In the fast limit $F$ goes $\infty$, the singular heteroclinic connecting
$(x_{1},0)$ to $(x_{0},0)$ closes if and only if
$$
x_{e}=x_{e}^{\ast}
:=\frac{2 x_{0}x_{1}}{x_{0}+x_{1}} .
$$
\end{pro}

\section{Cyclicity of a hyperbolic heteroclinic connection between linear saddles with an invariant line}\label{sec:cyclicityheterocline}

To determine the cyclicity of the heteroclinic connection, we study the number of small zeros of the displacement map $\Delta(\rho)$ as $\rho$ goes to $0$.  
This map is constructed in a neighbourhood of the heteroclinic connection, and each small zero of $\Delta(\rho)$ corresponds to a limit cycle bifurcating from the connection.

In this section, we focus on a particular case where both subsystems $Z_{1}$ and $Z_{2}$ are linear vector fields, and the heteroclinic connection is hyperbolic.  
These facts considerably simplify the analysis: the phase portraits and the level curves of the corresponding first integrals are known explicitly, allowing the semi-return maps to be computed in closed form.

If the connection were non-hyperbolic, the displacement map would be much harder to analyse.  
Fortunately, in the piecewise linear system \eqref{eq:pwlinear}, by Corollary~\ref{cor:unstableconnection}, only unstable hyperbolic heteroclinic connections can occur.

\medskip

Our approach follows the standard procedure:  
\begin{enumerate}[(i)]
    \item We compute the \emph{semi-return maps} on each side of the discontinuity line $\Sigma$.
    \item We study the \emph{displacement map} $\Delta(\rho)$, whose zeros correspond to limit cycles.
\end{enumerate}

For convenience, we perform a translation so that the heteroclinic connection, which originally occurs at $x = x_{e}$, now lies at the origin of coordinates.  
Let $I_{+}$ and $I_{-}$ denote the first integrals of the subsystems $\Sigma^{+}$ and $\Sigma^{-}$, respectively.  
The parameter value corresponding to the heteroclinic bifurcation is
$$
F = F_{h}(x_{e}) := \frac{-x_{0} - x_{1} + 2x_{e}}{x_{0}x_{1} - x_{e}^{2}},
$$
and the height of the heteroclinic connection is given by
$$
y_{h} := -\frac{2\,(x_{1}-x_{e})\left[\left(x_{1}-\tfrac{x_{e}}{2}\right)x_{0}-\tfrac{x_{1}x_{e}}{2}\right](x_{0}-x_{e})}
{x_{1}x_{0}\,\big(x_{0}x_{1}-x_{e}^{2}\big)} > 0.
$$

On the upper side of the discontinuity line, the return condition requires that the orbit reenters the section $\Sigma$ at the same energy level of $I_{+}$.  
This leads to
\begin{equation}    \label{eq:S1}
    S_1 := I_{+}(x=0, y=\rho) - I_{+}(x=0, y=y_{h}-r_{1}),
\end{equation}
where $r_{1}$ denotes the distance between the intersection point and the heteroclinic connection height after flowing along $\Sigma^{+}$.

Analogously, on the lower side of the discontinuity line, we impose the same condition for $I_{-}$:
\begin{equation}    \label{eq:S2}
    S_2 := I_{-}(x=0, y=\rho) - I_{-}(x=0, y=y_{h}-r_{2}),
\end{equation}
where $r_{2}$ denotes the ordinate at which the flow has to start to finish at $\rho$ flowing along $\Sigma^{-}$.

The idea is to study the equations $S_{1} = 0$ and $S_{2} = 0$, and to determine the functions 
$r_{1}(\rho)$ and $r_{2}(\rho)$ such that 
$S_{1}(\rho, r_{1}(\rho)) = 0$ and $S_{2}(\rho, r_{2}(\rho)) = 0$, respectively. 
The displacement map is then defined as
\[
\Delta(\rho) = r_{1}(\rho) - r_{2}(\rho),
\]
which is constructed from the two semi-return maps.

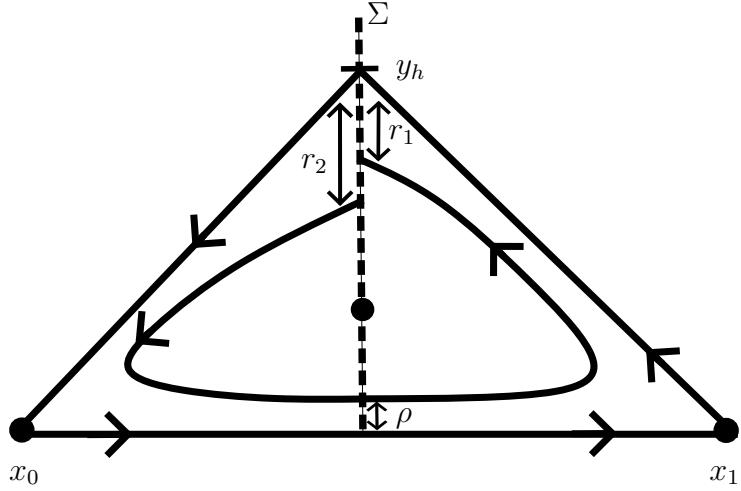
\begin{figure}[h!]
    \centering
    \begin{overpic}[width=0.6\textwidth]{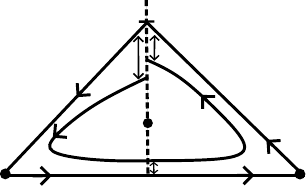}
        \put(53,5){$\rho$}
        \put(0,-3){$x_{0}$}
        \put(96,-3){$x_{1}$}
        \put(49,60){$\Sigma$}
        \put(53,53){$y_{h}$}
        \put(52,44){$r_{1}$}
        \put(40,40){$r_{2}$}
    \end{overpic}
    \caption{Return map near the heteroclinic connection. The limit cycle occurs when $r_{1}(\rho) = r_{2}(\rho)$.}
\end{figure}

To study the cyclicity, we focus on the region of parameter space where the heteroclinic connection exists, namely
$$
\sqrt{x_{0}x_{1}} < x_{e} < \frac{x_{0}+x_{1}}{2} .
$$
In this range, the following inequalities hold:
$$
(x_{0}x_{1} - x_{e}^{2}) < 0, \qquad 
\left[\left(x_{1}-\tfrac{x_{e}}{2}\right)x_{0} - \tfrac{x_{1}x_{e}}{2}\right] < 0,
$$
which ensure that the connection is hyperbolic and, as previously noted, unstable. These inequalities are useful for the study of the sign of the coefficient in the semi-return maps.

\subsection{Semi-return maps}

To construct the return map, we analyse separately the semi-return maps 
associated with the subsystems $\Sigma^{+}$ and $\Sigma^{-}$.
For each side, we obtain an implicit relation between the local coordinates
$(\rho, r_{i})$, for $i\in{1,2}$, which will then be used to compute the displacement map.

\subsubsection{Map $S_{1}$}

From equation~\eqref{eq:S1}, we obtain
\begin{equation*}
    \frac{(A_{1,1}\rho - A_{1,0})B_{1}}{C_{1}} 
    = \frac{r_{1}D_{1}(E_{0,1} - E_{1,1}r_{1})^{\upsilon}}{F_{1}},
\end{equation*}
or, equivalently,
\begin{equation}\label{eq:retornheteroclina1}
    (A_{1,1}\rho - A_{1,0})z_{1} 
    = -\,r_{1}D_{1}(E_{0,1} - E_{1,1}r_{1})^{\upsilon},
\end{equation}
where the constants are chosen to be positive.

\medskip
\noindent
The exponent $\upsilon$ is given by
\[
\upsilon := 
-\frac{(x_{1} - x_{0})(x_{0}x_{1} - x_{e}^{2})}
{(x_{0}+x_{1}-2x_{e})(x_{1}-x_{e})x_{0}} < 0,
\]
and we introduce the auxiliary variable $z_{1} = \rho^{\upsilon}$.

We seek $r_{1}(\rho, z_{1})$ as a power series in $\rho$ and $z_{1}$,
\[
r_{1}(\rho, z_{1}) = \sum_{i,j \ge 0} a_{i,j} \rho^{i} z_{1}^{j},
\]
and determine the coefficients recursively from \eqref{eq:retornheteroclina1}.
At the first orders we obtain
\[
\begin{aligned}
a_{10} = 0, \quad 
a_{01} = \frac{A_{1,0}}{D_{1}E_{0,1}^{\upsilon}},\quad a_{20} = 0, \quad 
a_{11} = -\frac{A_{1,1}}{D_{1}E_{0,1}^{\upsilon}},\quad a_{02} = \frac{A_{1,0}^{2}\upsilon}{D_{1}^{2}E_{0,1}^{2\upsilon+1}}.
\end{aligned}
\]

Substituting these coefficients, the semi-return map on the $\Sigma^{+}$ side becomes
\[
S_{1}(\rho) =
\frac{\big(-D_{1}E_{0,1}^{-\upsilon}(A_{1,1}\rho - A_{1,0})
+ A_{1,0}^{2}E_{0,1}^{-1-2\upsilon}\upsilon\, \rho^{\upsilon}\big)\, \rho^{\upsilon}}
{D_{1}^{2}}+h.o.t.
\]

\subsubsection{Map $S_{2}$}

Analogously, from equation~\eqref{eq:S2} we have
\begin{equation*}
    \frac{(A_{2,1}\rho - A_{2,0})B_{2}}{C_{2}}
    = \frac{r_{2}D_{2}(E_{0,2}-E_{1,2}r_{2})^{\beta}}{F_{2}},
\end{equation*}
which can be rewritten as
\begin{equation*}
    (A_{2,1}\rho - A_{2,0})B_{2} 
    = r_{2}D_{2}(E_{0,2}-r_{2})^{\beta}+h.o.t.
\end{equation*}

The exponent $\beta$ is related to $\upsilon$ through
\[
\beta := \upsilon \frac{x_{0}}{x_{1}}\frac{x_{1}-x_{e}}{x_{e}-x_{0}}<0,
\]
and we define $z_{2} = \rho^{\upsilon\,\frac{x_{0}}{x_{1}}\frac{x_{1}-x_{e}}{x_{e}-x_{0}}}$.

Assuming the expansion
\[
r_{2}(\rho, z_{2}) = \sum_{i,j \ge 0} b_{i,j} \rho^{i} z_{2}^{j},
\]
we find
\[
\begin{aligned}
b_{10} = 0, \quad b_{01} = \frac{A_{2,0}}{D_{2}E_{0,2}^{\beta}}, \quad b_{20} = 0, \quad b_{11} = -\frac{A_{2,1}}{D_{2}E_{0,2}^{\beta}},\quad
b_{02} = \frac{A_{2,0}^{2}\beta}{D_{2}^{2}E_{0,2}^{2\beta+1}}.
\end{aligned}
\]

Hence, the semi-return map on the $\Sigma^{-}$ side is
\[
S_{2}(\rho) =
\frac{\big(-D_{2}E_{0,2}^{-\beta}(A_{2,1}\rho - A_{2,0})
+ A_{2,0}^{2}E_{0,2}^{-1-2\beta}\beta\, \rho^{\beta}\big)\, \rho^{\beta}}
{D_{2}^{2}}+h.o.t.
\]

\subsection{Displacement map}

The relation between the exponents is
\[
\beta = \upsilon \frac{x_{0}}{x_{1}}\frac{x_{1}-x_{e}}{x_{e}-x_{0}},
\qquad
\frac{x_{0}}{x_{1}}\frac{x_{1}-x_{e}}{x_{e}-x_{0}} < 1
\quad \Longrightarrow \quad
\beta < \upsilon < 0.
\]
Hence, both $\upsilon$ and $\beta$ are negative in the parameter region where the heteroclinic connection and its bifurcation occur.

\medskip
The displacement map is defined as
\begin{align*}
\Delta(\rho) &= r_{2}(\rho) - r_{1}(\rho) \\[3pt]
&= 
\left(
-\frac{E_{0,1}^{-\upsilon} A_{1,1} \rho^{\upsilon}}{D_{1}}
+ \frac{E_{0,2}^{-\beta} A_{2,1} \rho^{\beta}}{D_{2}}
\right) \rho
+ \frac{\upsilon A_{1,0}^{2}E_{0,1}^{-1-2\upsilon} \rho^{2\upsilon}
+ \rho^{\upsilon} D_{1} E_{0,1}^{-\upsilon} A_{1,0}}{D_{1}^{2}}\\
&\quad
- \frac{A_{2,0}\left( A_{2,0}\beta \rho^{2\beta} E_{0,2}^{-1-2\beta}
+ D_{2} \rho^{\beta} E_{0,2}^{-\beta} \right)}{D_{2}^{2}}
+ h.o.t.
\end{align*}

\medskip
Since $\beta<0$, we introduce the new variable
$$
R=\rho^{\beta},
$$
so that $R\to 0^{+}$ as $\rho\to 0^{+}$.  
This change of variables is smooth and one-to-one on $(0,\varepsilon)$, for any 
$0<\varepsilon\le \varepsilon_{0}$, and therefore zeros of $\Delta(\rho)$ correspond
exactly to zeros of $\Delta(R)$.

Under this transformation, the powers of $\rho$ satisfy
$$
\rho^{\beta}=R, \qquad
\rho^{\upsilon}=R^{f_\lambda}, \qquad
f_\lambda := \frac{\upsilon}{\beta} > 1.
$$

Substituting these relations into $\Delta(\rho)$ gives the following expansion in 
the variable $R$:
$$
\Delta(R)
=
\frac{\upsilon A_{1,0}^{2} E_{0,1}^{-1-2\upsilon}\, R^{2f_\lambda}
+ D_{1}E_{0,1}^{-\upsilon}A_{1,0}\, R^{f_\lambda}}{D_{1}^{2}}
-
\frac{A_{2,0}\left(
A_{2,0}\beta\,E_{0,2}^{-1-2\beta}R^{2}
+D_{2}E_{0,2}^{-\beta}R\right)}{D_{2}^{2}}
+ \text{h.o.t.}
$$

Collecting the dominant contributions yields
$$
\Delta(R)
=
R\left(
-\,\frac{A_{2,0}E_{0,2}^{-\beta}}{D_{2}}
+
R^{f_\lambda-1}\,\frac{E_{0,1}^{-\upsilon} A_{1,0}}{D_{1}}
\right)
+ o(R).
$$

Since  $\frac{A_{2,0}E_{0,2}^{-\beta}}{D_{2}}\neq0$, at most one small limit cycle can be unfolded. 

\subsection{Stability of the limit cycle}

The zeros of $\Delta(R)$ for small $R$ correspond to limit cycles near the 
heteroclinic connection.  
Since $\Delta(R)=\Pi(R)-R$, where $\Pi(R)$ is the return map, the stability 
of such a limit cycle is determined by
$$
\Pi'(R)
= 1 + \Delta'(R)
= 1
-\,\frac{A_{2,0}E_{0,2}^{-\beta}}{D_{2}}
+ (f_{\lambda}-1)\frac{E_{0,1}^{-\upsilon}A_{1,0}}{D_{1}}\,R^{f_{\lambda}-2}
+ \text{h.o.t.}
< 1,
$$
for $R>0$ sufficiently small.  Can be deduced that $0<\frac{A_{2,0}E_{0,2}^{-\beta}}{D_{2}}<1$ therefore, the limit cycle that bifurcates from the heteroclinic connection 
is \emph{stable}.

\medskip
\noindent
\textbf{Remark.}  
The exponent $f_{\lambda}=\upsilon/\beta$ appearing in the PWL model 
coincides with the exponent of the original model because, in both 
systems it is generated by the same ordered composition of hyperbolic sectors. The global transition map between saddles in the original model only affects the constant multiplying the leading term.

\subsection{Coefficients of the displacement map}
\vspace{0.5cm}

\renewcommand{\arraystretch}{1.5} 
\setlength{\tabcolsep}{3pt} 

\begin{center}
\begin{minipage}{0.48\textwidth}
\centering
\begin{tabular}{|c|p{5.6cm}|}
\hline
\textbf{Coeff.} & \textbf{Expression} \\
\hline\hline
$A_{1,0}$ & $-\frac{(x_{1}-x_{e})(x_{0}-x_{e})(2x_{1}x_{0}-x_{0}x_{e}-x_{1}x_{e})}{x_{1}x_{0}-x_{e}^{2}}$\\
$A_{1,1}$ & $x_{0}x_{1}$\\
$B_{1}$   & $z_{1}$\\
$C_{1}$   & $-\frac{x_{1}x_{0}-x_{e}^{2}}{(2x_{1}x_{0}-x_{0}x_{e}-x_{1}x_{e})(x_{0}-x_{e})}$\\
$D_{1}$   & $x_{0}x_{1}$\\
$E_{0,1}$ & $-\frac{2(x_{1}-x_{e})(x_{0}-x_{e})\left(\left(x_{1}-\frac{x_{e}}{2}\right)x_{0}-\frac{x_{1}x_{e}}{2}\right)}{x_{0}x_{1}(x_{1}x_{0}-x_{e}^{2})}$\\
$E_{1,1}$ & 1\\
$F_{1}$   & $-C_{1}$\\
$\upsilon$  & $\frac{(-x_{1}+x_{0})(x_{1}x_{0}-x_{e}^{2})}{(x_{0}+x_{1}-2x_{e})(x_{1}-x_{e})x_{0}}$\\
$z_{1}$   & $\rho^{\upsilon}$\\
\hline
\end{tabular}\\[0.5em]
\textbf{Semi-return map $S_{1}$}
\end{minipage}
\hfill
\begin{minipage}{0.48\textwidth}
\centering
\begin{tabular}{|c|p{5.6cm}|}
\hline
\textbf{Coeff.} & \textbf{Expression} \\
\hline\hline
$A_{2,0}$ & $-\frac{(x_{1}-x_{e})(x_{0}-x_{e})(2x_{1}x_{0}-x_{0}x_{e}-x_{1}x_{e})}{x_{1}x_{0}-x_{e}^{2}}$\\
$A_{2,1}$ & $x_{0}x_{1}$\\
$B_{2}$   & $z_{2}$\\
$C_{2}$   & $-\frac{x_{1}x_{0}-x_{e}^{2}}{(x_{1}-x_{e})(2x_{1}x_{0}-x_{0}x_{e}-x_{1}x_{e})}$\\
$D_{2}$   & $x_{0}x_{1}$\\
$E_{0,2}$ & $-\frac{2(x_{1}-x_{e})(x_{0}-x_{e})\left(\left(x_{1}-\frac{x_{e}}{2}\right)x_{0}-\frac{x_{1}x_{e}}{2}\right)}{x_{0}x_{1}(x_{1}x_{0}-x_{e}^{2})}$\\
$E_{1,2}$ & 1\\
$F_{2}$   & $-C_2$\\
$\beta$   & $\upsilon \frac{x_{0}}{x_{1}}\frac{x_{1}-x_{e}}{x_{e}-x_{0}}$\\
$z_{2}$   & $\rho^{\beta}$\\
\hline
\end{tabular}\\[0.5em]
\textbf{Semi-return map $S_{2}$}
\end{minipage}
\end{center}

\vspace{1cm}

\end{document}